\documentclass[11pt, reqno]{article}
\usepackage[english]{babel}
\usepackage[utf8]{inputenc}
\usepackage{authblk}
\usepackage{latexsym,amsfonts,amssymb,amsthm,mathrsfs,amsmath,amscd,enumerate, enumitem,esint}
\usepackage[usenames,dvipsnames,svgnames]{xcolor} %black, blue, brown, cyan, darkgray, gray, green, lightgray, lime, magenta, olive, orange, pink, purple, red, teal, violet, white, yellow
 \newtheorem{thm}{Theorem}[section]
 
 \newtheorem{lem}[thm]{Lemma}
 
\theoremstyle{definition}
\newtheorem{defn}[thm]{Definition}
 \theoremstyle{remark}
 \newtheorem{rem}[thm]{Remark}

\newcommand{\abs}[1]{\left\vert#1\right\vert}

\newcommand{\norm}[1]{\left\|#1\right\|}

\numberwithin{equation}{section}
\allowdisplaybreaks
\begin{document}

    \title{The effect of the smoothness of fractional type operators over their commutators with Lipschitz symbols on weighted spaces}

\author[1]{Estefanía Dalmasso\thanks{edalmasso@santafe-conicet.gov.ar}}
\author[2]{Gladis Pradolini\thanks{gladis.pradolini@gmail.com}}
\author[3]{Wilfredo Ramos\thanks{wramos@santafe-conicet.gov.ar}}
\affil[1]{{\footnotesize Instituto de Matemática Aplicada del
Litoral, UNL, CONICET, FIQ, Santa Fe, Argentina.}}
\affil[2]{{\footnotesize Facultad de Ingenier\'ia Qu\'imica, UNL,
Santa Fe, Argentina. Researcher of CONICET.}}
\affil[3]{{\footnotesize Facultad de Ciencias Exactas y Naturales y
Agrimensura, UNNE, Corrientes, Argentina. Researcher of CONICET.}}

\date{\vspace{-1.5cm}}

\maketitle

  \begin{abstract}
   We prove boundedness results for integral operators of fractional type and their higher order commutators between weighted spaces, including $L^p$-$L^q$, $L^p$-$BMO$ and
    $L^p$-Lipschitz estimates. The kernels of such operators satisfy certain size condition and
    a Lipschitz type regularity, and the symbol of the commutator belongs to a Lipschitz class. We also deal with commutators of fractional type operators with less regular kernels satisfying a
    H\"ormander's type inequality. As far as we know, these last results are new even in the unweighted case. Moreover, we give a characterization result involving symbols of the commutators and continuity
    results for extreme values of $p$.
  \end{abstract}

\section{Introduction}

There is a close relationship between the theory of Partial
Differential Equations and Harmonic Analysis. This is evidenced, for
instance, by the existence of a mechanism that  provides us with regular solutions of {\sl PDE's} when we equip this
machinery with continuity properties of certain related operators
(see, for example, \cite{BCe}, \cite{BCM}, \cite{ChFL},
\cite{ChFL2}, \cite{DR}, \cite{Rios}). Therefore, it seems
appropriate to explore the boundedness properties of the mentioned
operators and, particularly, we shall be concerned with commutators
of integral operators of fractional type.

It is well known that the fractional integral operator of order
$\alpha$, $0<\alpha<n$, is defined by
\begin{equation*}
I_{\alpha}f(x)=\int_{\mathbb{R}^{n}}\frac{f(y)}{|x-y|^{n-\alpha}}\,
dy,
\end{equation*}
whenever this integral is finite. There is a vast amount of
information about the behavior of the operator above (see for
example \cite{HSV}, \cite{HL}, \cite{MW}, \cite{Pra} and
\cite{Swe}). In \cite{Chanillo}, Chanillo introduced the first order
commutator of $I_{\alpha}$  with a symbol $b\in L_{\rm
    loc}^{1}(\mathbb{R}^{n})$, formally defined by
\begin{equation*}
[b,I_{\alpha}](f)= b\,I_{\alpha}f-I_{\alpha}(bf).
\end{equation*}
Particularly, if $b\in BMO$, the space of bounded mean oscillation,
the author proved that, for $1<p<n/\alpha$ and $1/q=1/p-\alpha/n$,
the operator $[b,I_{\alpha}]$ is bounded from $L^{p}(\mathbb R^n)$ into
$L^{q}(\mathbb R^n)$. Some topics related to properties of boundedness for
commutators of fractional integral operators for extreme values of
$p$ can be found in \cite{HST}. (For more information about $BMO$
spaces see \cite{JN1}).

The continuity properties of the commutator of the fractional
integral operator acting on different spaces were studied by several
authors contributing, in this way, to the development of {\sl
    PDE's}. Some related articles are given by \cite{BHP}, \cite{CUF}, \cite{DLZ}, \cite{GPS4}, \cite{LK}, \cite{P}, \cite{P2},
\cite{PPTT}, \cite{PS}, \cite{ST}. In \cite{MY} the authors consider
the commutators of certain fractional type operators with Lipschitz
symbols and prove the boundedness between Lebesgue spaces,
including the boundedness from Lebesgue spaces into $BMO$ and
Lipschitz spaces on non-homogeneous spaces. (See also \cite{PR2} in
the context of variable Lebesgue spaces).

Nevertheless, there is not enough information about the behavior of
the commutators acting between weighted Lebesgue spaces, even less
for extreme values of $p$, that is, the weighted $L^{p}$-$BMO$ or
$L^{p}$-{\sl Lipschitz} boundedness. Hence, one of our main aims is,
precisely, to give sufficient conditions on the weights in order to
obtain these continuity properties. Some previous results in this
direction were given in \cite{BDP} where the authors study the
boundedness between Lebesgue spaces with variable exponent for
commutators of fractional type operators with $BMO$ symbols,
(see also \cite{DP} in the framework of Orlicz spaces).

%2) Decir donde se estudia el problema con pesos de los operadores
%(mi artículo de int fraccionaria y algún otro) y  sin pesos para
%conmutadores asociados con simbolos Lipschitz (MY, Pota, Wil, cuál
%otro?)

%3) Insistir en que no hay nada con pesos para los conmutadores,
%menos aún con las acotaciones Lp lip.

%4) Si nos salen las Hormander aunque sea sin pesos (Lp-Lip y
%caracterización (ver verso de Pola para esta parte)) dar la
%importancia de estos operadores (Ana B et al) y ejemplos rápidamente
%y contar resultados. Decir que son buenos aportes ya que no
%conocemos nada de ellos, sólo lo de Ana y artículo de Estefi con
%pesos y simbolos en BMO.

%5) Mostrar bien claro las condiciones que se piden en el tamaño del
%núcleo y en la suavidad en cada caso (aunque creo que a eso lo
%tenemos pero me gustaría insistir acá).

We shall first consider fractional type operators, and their
commutators, which kernels satisfy certain size condition and a
Lipschitz type regularity. For this type of operators we prove
boundedness results of the type described above. Particularly we
prove a characterization result involving symbols of the commutators
and continuity results for extreme values of $p$.

Later, we study commutators of fractional type operators with less
regular kernels. These type of operators include a great variety of
operators and were introduced in \cite{BLR}. See section
\ref{subsechormander} for examples and  more information.

The paper is organized as follows. In section \S
$\ref{preliminaries}$ we give the preliminaries definitions in order
to state the main results of the article, which are also included in
this section. Then, in \S $\ref{Aux}$ we give some auxiliary results
which allow us to prove the main results in \S $\ref{proofs}$.

\section{Preliminaries and main results}\label{preliminaries}

In this section we give the definitions of the operators we shall be
dealing with and the functional class of the symbols in order to
define the commutators.

We shall consider fractional operators of convolution type
$T_\alpha$, $0< \alpha<n$, defined by
\begin{equation}\label{def1}
T_\alpha f(x)=\int_{\mathbb{R}^{n}}K_\alpha(x-y)f(y)dy,
\end{equation}
where the kernel $K_\alpha$ is not identically zero and verifies
certain size and smoothness conditions.

Let $0< \delta < 1$. We say that a function $b$ belongs to the space
$\Lambda (\delta)$ if there exists a positive constant $C$ such
that, for every $x$, $y\in \mathbb R^n$
%\begin{equation}\label{Lipnorm}
%\sup\limits_{B}\frac{1}{|B|^{1+\frac{\delta}{n}}} \int_{B} |b(x)-b_B| dx <\infty,
%\end{equation}
%where the supremum is taken over all balls $B\subset \mathbb R^n$ and $b_B$ denotes the average $\fint_{B} b(x) dx$. The space $\mathbb L(\delta)$ coincides with the classical space $\Lambda(\delta)$ consisting on the functions $b$ such that
\[|b(x)-b(y)|\leq C|x-y|^\delta.\]
 The smallest of such constants will be denoted by $\|b\|_{\Lambda
(\delta)}$. We will be dealing with commutators with symbols
belonging to this class of functions.

Given a weight $w$, that is, a non-negative and locally integrable
function, we say that a measurable function $f$ belongs to
$L^p_w(\mathbb R^n)$ for some $1< p< \infty$ if $fw\in L^p(\mathbb
R^n)$

We classify the operators defined in \eqref{def1} into two different
types, according to the conditions satisfied by $K_\alpha$.

\subsection{Fractional integral operators with Lipschitz regularity}

    We say that $K_\alpha$ satisfies the size condition $S_\alpha$ if it verifies the following inequality
        \begin{equation*}
        \int_{s<|x|\leq 2s} |K_{\alpha}(x)|\, dx \leq C s^{\alpha},
        \end{equation*}
        for every $s>0$ and some positive constant $C$.

    We shall also assume that $K_\alpha$ satisfies the smoothness condition $H_{\alpha,\infty}^*$, that is, there exist a positive constant $C$ and $0< \eta\leq 1$ such that
    \begin{equation*}
    \abs{K_\alpha(x-y)-K_\alpha(x'-y)}+\abs{K_\alpha(y-x)-K_\alpha(y-x')}\leq
    C\frac{\abs{x-x'}^{\eta}}{\abs{x-y}^{n-\alpha +\eta}},
    \end{equation*}
    whenever $\abs{x-y}\geq 2\abs{x-x'}$.

A typical example is the fractional integral operator $I_\alpha$, whose kernel $K_\alpha(x)=|x|^{\alpha-n}$ satisfies
conditions $S_{\alpha}$ and $H_{\alpha,\infty}^*$, as it can be
easily checked.

    Related with the fractional integral operators $T_\alpha$, we can formally define the higher order commutators with symbol $b\in L^1_{{\rm loc}}(\mathbb R^n)$, by
    \begin{equation*}\label{commutators}
    T_{\alpha,b}^m f(x)=\int_{\mathbb{R}^{n}}(b(x)-b(y))^m K_\alpha(x-y)f(y)dy,
    \end{equation*}
where $m\in \mathbb N_0$ is the order of the commutator. Clearly,
$T_{\alpha,b}^0=T_\alpha$.

As we have said, we are interested in studying the boundedness
properties of the commutators $T_{\alpha, b}^m$, with symbol $b\in
\Lambda(\delta)$, on weighted spaces. We shall first consider their
continuity on weighted Lebesgue spaces of the type defined previously. We
shall also analyze the boundedness of $T_{\alpha, b}^m$ from
weighted Lebesgue spaces into certain weighted version of Lipschitz
spaces. For $0<\delta<1$ and $w$ a weight, these spaces are denoted by $\mathbb L_w(\delta)$ and collect the functions $f\in
L^1_{{\rm loc}}(\mathbb R^n)$ that satisfy
\begin{equation}\label{Lipnormw}
\sup\limits_{B}\frac{\|w\chi_B\|_\infty}{|B|^{1+\frac{\delta}{n}}} \int_B |f(x)-f_B| dx<\infty,
\end{equation}
where $\|g\|_\infty$ denotes the essential supremum of a measurable
function $g$. The case $\delta=0$ of the space above was introduced
in \cite{MW} as a weighted version of the space of functions with
bounded mean oscillation.

The classes of weights we will be dealing with are the well-known
$A_{p,q}$ classes of Muckenhoupt and Wheeden (\cite{MW}). For $1\leq
p,q< \infty$ these classes are defined as the weights $w$ such
that
\[[w]_{A_{p,q}}:=\sup\limits_{B} \left(\frac{1}{|B|}\int_B w(x)^q dx\right)^{1/q}\left(\frac{1}{|B|}\int_B w(x)^{-p'} dx\right)^{1/p'}<\infty.\]
When $q=\infty$, we understand that $w\in A_{p,\infty}$ as $w^{-p'}
\in A_1$.

In this subsection we shall assume that the operator $T_\alpha$ has
a convolution kernel $K_\alpha$ that verifies conditions
$S_{\alpha}$ and $H_{\alpha,\infty}^*$ with $0<\eta\leq 1$. In order
to simplify the hypothesis we shall supposse that $m\in \mathbb N
\cup\{0\}$ with the convention that $\beta/0=\infty$ if $\beta>0$.

We now give the boundedness result between weighted Lebesgue spaces
for the higher order commutators of $T_{\alpha}$ with Lipschitz
symbols.

\begin{thm}\label{teo1}Let $0<\alpha<n$ and
$0<\delta<\min\{\eta,(n-\alpha)/m\} $. Let $1<p<n/(m\delta+\alpha)$,
$1/q=1/p-(m\delta+\alpha)/n$ and $b\in \Lambda(\delta)$. If $w\in
A_{p,q}$, then there exists a positive constant $C$ such that
\[\left(\int_{\mathbb R^n} |T_{\alpha,b}^m f(x)|^q w(x)^q dx
\right)^{1/q}\leq C \norm{b}_{\Lambda(\delta)}^m \left(\int_{\mathbb
R^n} |f(x)|^p w(x)^p dx\right)^{1/p}\] for every $f\in L^p_w(\mathbb
R^n)$.
    \end{thm}

\begin{rem}
When $m=0$ and $T_{\alpha}=I_{\alpha}$, the result above was proved
in \cite{MW}.  Notice that there are no symbols or parameters
$\delta$ in the hypothesis in this case.
\end{rem}

%\smallskip

The next result gives the continuity properties of $T_{\alpha,b}^{m}$
between weighted Lebesgue spaces and
$\mathbb{L}_{w}(\widetilde{\delta})$ spaces. By $\beta^+$ we understand $\beta$ if $\beta>0$ and $0$ is $\beta\leq 0$.

{ \begin{thm}\label{boundednessLipw}Let $0<\alpha<n$, and
$0<\delta<\min\{\eta,(n-\alpha)/m\} $. Let $n/(m\delta+\alpha)\le
r<n/(\alpha+(m-1)\delta)$, if $m\ge 1$ or $n/\alpha\le
r<n/(\alpha-\eta)^+$, if $m=0$. Let $\widetilde{\delta}=m\delta
+\alpha-\frac{n}{r}$ and $b\in\Lambda(\delta)$. If $w\in
A_{r,\infty}$, then there exists a positive constant $C$ such that
   \begin{equation*}
     \norm{T_{\alpha,b}^{m}f}_{\mathbb{L}_{w}(\widetilde{\delta})}\leq C \norm{b}_{\Lambda(\delta)}^m\norm{fw}_{L^{r}}
   \end{equation*}
  for every $f\in L^r_w(\mathbb R^n)$.
 \end{thm}
}

\smallskip

\begin{rem}
When $m=1$, $w=1$ and $T_{\alpha}=I_{\alpha}$, the result above was proved in \cite{MY} in the general context of non-doubling measures.
\end{rem}
\begin{rem}
If $r=n/(m\delta+\alpha)$, then $\widetilde{\delta}=0$ and the space
$\mathbb{L}_{w}(\widetilde{\delta})$ is the weighted version of
the $BMO$ spaces introduced in \cite{MW}. By taking into account the
range of $p$ in Theorem $\ref{teo1}$, it is clear that this is the
endpoint value from which the Lebesgue spaces change into $BMO$ and
Lipschitz spaces when the commutator acts. Particularly, if $m=0$
and $T_{\alpha}=I_{\alpha}$, this is the well-known result proved in
\cite{MW}. Notice again that there are no parameters $\delta$ or
symbols in this case. \\On the other hand, if $m=0$,
$T_{\alpha}=I_{\alpha}$ and $n/\alpha\le r <n/(\alpha-1)^+$, that is,
$\eta=1$ in the definition of the class $H_{\alpha, \infty}^*$, the
result above was proved in \cite{Pradopolonia}.
\end{rem}

%\begin{rem}Notice that the hypothesis in the theorem above allows
%$\widetilde{\delta}$ to be negative, which means that $\mathbb
%L_w(\widetilde{\delta})$ is a weighted space in the scale of Morrey
%spaces.
%\end{rem}

For the extreme value $r=n/((m-1)\delta+\alpha)$, $m\in \mathbb N$
and $0< \delta<\eta\leq1$, we obtain the following endpoint result
in order to characterize the symbol $b$ in $\Lambda(\delta)$ in
terms of the boundedness of $T_{\alpha,b}^m$ in the sense of Theorem
\ref{boundednessLipw}. In order to give this result we introduce
some previous notation. For $k=0,1,...,m$ we denote
$c_{k}={m!}/{(k!(m-k)!)}$. If also $x$, $u\in \mathbb R^n$, we denote $S
(x,u,k)=(b(x)-b_{B})^{m-k}T_{\alpha}((b-b_{B})^{k}f_{2})(u)$, where $f_2=f\chi_{\mathbb R^n\setminus B}$ for a given ball $B$ and a locally integrable function $f$.

\begin{thm}\label{limitcase}Let $m\in \mathbb N$, $0< \delta<\min\{\eta,(n-\alpha)/m\}$ and $r=n/((m-1)\delta+\alpha)$.
If $w\in A_{n/(m\delta+\alpha), \infty}$ and $b\in \Lambda(\delta)$,
the following statements are equivalent.
\begin{enumerate}[label=(\roman*), leftmargin=0cm,itemindent=.6cm,labelwidth=\itemindent,labelsep=0cm,align=left]
        \item \label{ecu1-m>1} $T_{\alpha,b}^{m}:L^r_w(\mathbb R^n)(\mathbb R^n)\hookrightarrow \mathbb{L}_{w}(\delta)$;
        \item There exists a positive constant $C$ such that
        \begin{equation}\label{ecu2-m>1}
        \frac{\norm{w\chi_B}_{\infty}}{|B|^{1+\frac{\delta}{n}}}
        \int_B\abs{\sum_{k=0}^{m}c_{k}\left[S(x,u,k)-(S(\cdot,u,k))_{B}\right]}dx\leq C \norm{fw}_{r},
        \end{equation}
        for every ball $B\subset \mathbb{R}^{n}$, $x, u\in B$ and $f\in L^{1}_{{\rm loc}}(\mathbb
        R^n)$.
\end{enumerate}
\end{thm}

\begin{rem}
When $m=1$, $w=1$ and $T_{\alpha}=I_{\alpha}$ the result above was
proved in \cite{MY} in a more general context of non-homogeneous
spaces. Certainly, their result was inspired in the article of
\cite{HST}, where the same result is proved for $m=1$, $w=1$, $T_{\alpha}=I_{\alpha}$ and $b\in BMO$.
\end{rem}

\begin{rem}
In \cite{HST} the authors also have obtained that, in the case of $T_{\alpha}=I_{\alpha}$, $b\in BMO$, $m=1$ and $w=1$, the boundedness of the commutator $I_{\alpha,b}$ from $L^{n/\alpha}$ into $BMO$ can only occur if $b$ is constant. In our case, if $b\in \Lambda(\delta)$, when  $T_{\alpha}=I_{\alpha}$, $m=1$ and $w=1$, we can deduce, by
\eqref{ecu2-m>1}, that, if $I_{\alpha,b}$ is bounded from $L^{n/\alpha}$ into $\mathbb{L}(\delta)$, then
\begin{equation*}
\frac{1}{|B|^{1+\frac{\delta}{n}}}
\int_B\abs{\sum_{k=0}^{1}c_{k}\left[S(x,u,k)-(S(\cdot,u,k))_{B}\right]}dx\leq C \norm{f}_{n/\alpha}.
\end{equation*}
Since it is easy to see that
\begin{equation*}
\frac{1}{|B|^{1+\frac{\delta}{n}}}
\int_B\abs{\sum_{k=0}^{1}c_{k}\left[S(x,u,k)-(S(\cdot,u,k))_{B}\right]}dx=\frac{1}{\abs{B}^{1+\frac{\delta}{n}}}\int_{B}\abs{b(x)-b_{B}}dx
\abs{\int_{(2B)^{c}}\frac{f(y)}{\abs{u-y}^{n-\alpha}}dy},
\end{equation*}
we have that
\begin{equation*}
\frac{1}{\abs{B}^{1+\frac{\delta}{n}}}\int_{B}\abs{b(x)-b_{B}}dx
\abs{\int_{(2B)^{c}}\frac{f(y)}{\abs{u-y}^{n-\alpha}}dy}\leq C \norm{f}_{n/\alpha}.
\end{equation*}

Following the same guidelines as in \cite{HST} with $f_{N}(y)=\abs{u-y}^{-\alpha}\chi_{B(0,N)}(u-y)\chi_{(2B)^{c}}(y)$ for $N\in\mathbb{N}$, we obtain that
\begin{equation*}
\frac{1}{\abs{B}^{1+\frac{\delta}{n}}}\int_{B}\abs{b(x)-b_{B}}dx
\abs{\int_{(2B)^{c}\cup \{\abs{u-y}<N\}}\frac{dy}{\abs{u-y}^{n}}}^{1-\alpha/n}\leq C.
\end{equation*}
Due tu the fact that  $\abs{\int_{(2B)^{c}\cup \{\abs{u-y}<N\}}\frac{dy}{\abs{u-y}^{n}}}^{1-\alpha/n}\to \infty$ when $N\to \infty$, we have  $b(x)=b_{B}$ almost everywhere, for every ball $B$, which yields that $b$ is essentially constant.
\end{rem}

\subsection{Fractional integral operators with H\"ormander type
regularity}\label{subsechormander}

We now introduce the conditions on the kernel that will be considered in this section. First, we must give some notation.

It is well-known that the commutators of fractional integral operators can be controlled, in some sense, by maximal type operators associated to Young functions. By a Young function we mean a function $\Phi:[0,\infty)\rightarrow [0,\infty)$ that
    is increasing, convex and verifies $\Phi(0)=0$ and $\Phi(t)\rightarrow \infty$ when
    $t\rightarrow \infty$. The $\Phi$-Luxemburg average over a ball $B$ is defined, for a locally integrable function $f$, by
    \begin{equation*}
    \|f\|_{\Phi,B}=\inf\left\{\lambda >0: \frac{1}{|B|}\int_B \Phi\left(\frac{|f(x)|}{\lambda}\right)dx\leq 1\right\}.
    \end{equation*}

    The maximal type operators that control the commutators involve these averages. More precisely, if $f\in
    L^1_{{\rm loc}}(\mathbb R^n)$ and $0< \alpha<n$, we define the
    fractional type maximal operator associated to a Young function
    $\Phi$, by
    \begin{equation}\label{maxfrac}
    M_{\alpha, \Phi}f(x)=\sup\limits_{B\ni x} |B|^{{\alpha}/{n}}
    \|f\|_{\Phi,B},
    \end{equation}
    where the supremum is taken over every ball $B$ that contains $x\in \mathbb R^n$.

    Given a Young function $\Phi$, the following H\"older's type inequality holds for every pair of measurable functions $f,g$
    \[\frac{1}{|B|}\int_{B} |f(x)g(x)| dx \leq 2 \|f\|_{\Phi,B}\|g\|_{\widetilde{\Phi},B},\]
    where $\widetilde{\Phi}$ is the complementary Young function of
    $\Phi$, defined by
    \begin{equation*}
    \widetilde{\Phi}(t)= \sup_{s>0}\{st-\phi(s)\}.
    \end{equation*}
    It is easy to see that $t\leq
    \Phi^{-1}(t)\widetilde{\Phi}^{-1}(t)\leq 2t$ for every $t>0$.

Moreover, given $\Phi,\Psi$ and $\Theta$ Young functions verifying that $\Phi^{-1}(t)\Psi^{-1}(t)\lesssim\Theta^{-1}(t)$ for every $t>0$, the following generalization holds
    \[\|fg\|_{\Theta,B}\lesssim \|f\|_{\Phi,B}\|g\|_{\Psi,B}.\]

We are now in position to define the smoothness condition on $K_\alpha$.

We say that $K_{\alpha}\in H_{\alpha, \Phi}$ if
there exist $c\ge 1$ and $C>0$ such that for every $y\in \mathbb
R^n$ and $R>c|y|$
\begin{equation*}
\sum\limits_{j=1}^\infty (2^j R)^{n-\alpha} \|\left(K_{\alpha}(\cdot-y)-K_{\alpha}(\cdot)\right)\chi_{|\cdot|\sim 2^j R}\|_{\Phi, B(0,2^{j+1}R)}  \leq C,
\end{equation*}
where $|\cdot|\sim s$ means the set $\{x\in \mathbb R^n: s<|x|\leq 2s\}$.

When $\Phi(t)=t^q$, $1\leq q<\infty$, we denote this class by
$H_{\alpha,q}$ and it can be written as
\begin{equation*}
\sum\limits_{j=1}^\infty (2^j R)^{n-\alpha}  \left(\frac{1}{(2^j
R)^n}\int_{|x|\sim 2^j R}|K_{\alpha}(x-y)-K_{\alpha}(x)|^q
dx\right)^{1/q} \leq C.
\end{equation*}

The kernels given above are, a priori, less regular than the kernel
of the fractional integral operator $I_\alpha$ and they have been
studied by several authors. For example, in \cite{Ku}, the author
studied fractional integrals given by a multiplier. If
$m:\mathbb{R}^n\rightarrow \mathbb{R}$ is a function, the multiplier
operator $T_m$ is defined, through the Fourier transform, as
$\widehat{T_mf}(\zeta)=m(\zeta) \widehat{f}(\zeta)$ for $f$ in the
Schwartz class. Under certain conditions on the derivatives of $m$,
the multiplier operator $T_m$ can be seen as the limit of
convolution operators $T_m^N$, having a simpler form. Their
corresponding kernels $K_\alpha^N$ belong to the class $S_\alpha\cap
H_{\alpha,r}$ with constant independent of $N$, for certain values
of $r>1$ given by the regularity of the function $m$ (see
\cite{Ku}).

Other examples of this type of operators are fractional integrals
with rough kernels, that is, with kernel $K_\alpha(x)=\Omega(x)|x|^{\alpha-n}$
where $\Omega$ is a function defined on the unit sphere $S^{n-1}$ of
$\mathbb{R}^n$, extended to $\mathbb{R}^n\setminus \{0\}$ radially.
The function $\Omega$ is an homogeneous function of degree $0$. In
\cite[Proposition~4.2]{BLR}, the authors showed that $K_\alpha\in
S_\alpha\cap H_{\alpha,\Phi}$, for certain Young function $\Phi$,
provided that $\Omega\in L^{\Phi}(S^{n-1})$ with
\[\int_0^1 \omega_\Phi(t) \frac{dt}{t}<\infty,\]
    where $\omega_\Phi$ is the $L^\Phi$-modulus of continuity given by
    \[\omega_\Phi(t)=\sup\limits_{|y|\leq t}||\Omega(\cdot+y)-\Omega(\cdot)||_{\Phi,S^{n-1}}<\infty,\]
 for every $t\geq 0$. This type of operators where also studied in \cite{CWW} and \cite{DL}.

As we said previously, we are interested in studying the higher
order commutators of $T_\alpha$. Since we are dealing with symbols of Lipschitz type, the smoothness condition associated to these commutators is defined as follows.

\begin{defn}Let $m\in \mathbb N_0$, $0<\alpha<n$, $0\leq \delta<\min\{1, (n-\alpha)/m\}$ and let $\Phi$ be a Young function. We say that $K_\alpha\in H_{\alpha,\Phi,m}(\delta)$ if
    \begin{equation*}
    \sum\limits_{j=1}^\infty (2^j)^{m\delta}(2^j R)^{n-\alpha} \|\left(K_{\alpha}(\cdot-y)-K_{\alpha}(\cdot)\right)\chi_{|\cdot|\sim 2^j R}\|_{\Phi, B(0,2^{j+1}R)}  \leq C.
    \end{equation*}
for some constants $c\ge 1$ and $C>0$ and for every $y\in \mathbb   R^n$ with $R>c|y|$.
\end{defn}

Clearly, when $\delta=0$ or $m=0$, $H_{\alpha,\Phi,m}(\delta)=H_{\alpha,\Phi}$.

\begin{rem}\label{contentionsHdelta}
    It is easy to see that $H_{\alpha,\Phi,m}(\delta_2)\subset H_{\alpha,\Phi,m}(\delta_1) \subset H_{\alpha,\Phi}$ whenever $0\leq \delta_1<\delta_2<\min\{1,(n-\alpha)/m\}$.
\end{rem}

Recall that Fourier multipliers and fractional integrals with rough kernels are examples of fractional integral operators with $K_\alpha\in H_{\alpha,\Phi}$ for certain Young function. By assuming adequate conditions depending on $\delta$ on the multiplier $m$, or on the $L^\Phi$-modulus of continuity $\omega_\Phi$, one can obtain kernels $K_\alpha\in H_{\alpha,\Phi,m}(\delta)$. This fact can be proved by adapting Proposition 4.2 and Corollary 4.3 given in \cite{BLR} (see also \cite{LMRT}).

We shall also deal with a class of Young functions that arises in connection with the boundedness of the
fractional maximal operator $M_{\alpha,\Psi}$ on weighted Lebesgue spaces (see \S\ref{Aux}). Given $0<\alpha<n$, $1\leq \beta<p<n/\alpha$ and a Young function $\Psi$, we
shall say that $\Psi\in \mathcal B_{\alpha,\beta}$ if
$t^{-\alpha/n}\Psi^{-1}(t)$ is the inverse of a Young function and
$\Psi^{1+\frac{\rho\alpha}{n}}\in B_{\rho}$ for every
$\rho>n\beta/(n-\alpha\beta)$, that is, there exists a positive constant $c$
such that
\[\int_c^\infty \frac{\Psi^{1+\frac{\rho\alpha}{n}}(t)}{t^\rho} \frac{dt}{t}<\infty\]
for each of those values of $\rho$.

We now state the following
generalizations of Theorems \ref{teo1} and \ref{boundednessLipw}. We shall consider again $m\in \mathbb N_0$.
%       Before introducing the following results, we shall define the functions
%       $C_{m,\delta}^{-1}(t)=1$ if $\delta>0$ and $C_{m,\delta}^{-1}(t)=\log(1+t)^m$ if $\delta=0$. Thus, in the last case, $C_{m,0}(t)=e^{t^{1/m}}-1$ and $\widetilde{C_{m,0}}(t)\thickapprox t\log(e+t)^m$. For the definition of the class $\mathcal Y_{\alpha,\beta}$, see \S\ref{Aux}.
%
\begin{thm}\label{teo1H} Let $0<\alpha<n$ and
$0< \delta<\min\{1,(n-\alpha)/m\} $. Let $1<p<n/(m\delta+\alpha)$,
$1/q=1/p-(m\delta+\alpha)/n$ and $b\in \Lambda(\delta)$. Assume that $T_\alpha$ has a
kernel $K_\alpha\in S_\alpha\cap H_{\alpha, \Phi}$ for a Young
function $\Phi$ such that its complementary function
$\widetilde{\Phi}\in \mathcal B_{m\delta+\alpha,\beta}$ for some $1\leq \beta<p$. Then, if $w$ is
a weight verifying $w^\beta \in A_{\frac{p}{\beta},\frac{q}{\beta}}$, there
exists a positive constant $C$ such that
\[\left(\int_{\mathbb R^n} |T_{\alpha,b}^m f(x)|^q w(x)^q dx \right)^{1/q}\leq C \|b\|_{\Lambda(\delta)}^m
\left(\int_{\mathbb R^n} |f(x)|^p w(x)^p dx\right)^{1/p}\]
for every $f\in L^p_w(\mathbb
R^n)$.
        \end{thm}

\begin{rem}If we consider, for example, $\Phi(t)=e^{t^{1/\gamma}}-e$ with $\gamma>0$, then $\widetilde{\Phi}(t)\thickapprox t(1+\log^+ t)^\gamma$ and this function verifies condition $\mathcal B_{m\delta+\alpha,1}$. Thus, $\Phi$ satisfies the hypothesis of the theorem above and, in this case, we can take $w\in A_{p,q}$. As we have mentioned before, this condition $\mathcal B_{m\delta+\alpha,\beta}$ is related with the boundedness of the corresponding fractional maximal operator $M_{m\delta+\alpha, \widetilde{\Phi}}$ between $L^p_w$ and $L^q_w$ when $w^\beta\in A_{\frac{p}{\beta},\frac{q}{\beta}}$ (see Theorem \ref{BDPcteH} below). When $\beta>1$, a typical example is $\widetilde{\Phi}(t)=t^\beta(1+\log^+t)^\gamma$ for $\gamma\geq 0$.  In this case, the Young function $\Phi$ related with the smoothness condition on the kernel $K_\alpha$ given in the theorem above is $\Phi(t)=t^{\beta'}(1+\log^+ t)^{-\gamma/(\beta-1)}$, where $\beta'=\beta/(\beta-1)$.
    \end{rem}

\begin{thm}\label{HormanderLipw}Let $0<\alpha <n$,
$0< \delta<\min\{ 1,(n -\alpha)/m\}$ and $n/(m\delta+\alpha)\leq r<n/((m-1)\delta+\alpha)$ such that $\widetilde{\delta}=m\delta
+\alpha-\frac{n}{r}$. Let $w$ be a weight such that $w^\beta\in A_{r/\beta,\infty}$ for some $1<\beta<r$. Assume
that $T_\alpha$ has a kernel $K_\alpha\in S_\alpha\cap H_{\alpha,
\Phi,m}(\delta)$ for a Young function $\Phi$ such that $\Phi^{-1}(t)\lesssim t^{\frac{\beta-1}{r}}$ for every $t>0$. If
$b\in\Lambda(\delta)$, then there exists a positive constant $C$
such that
    \begin{equation*}
    \norm{T_{\alpha,b}^{m}f}_{\mathbb{L}_{w}(\widetilde{\delta})}\leq C \norm{b}_{\Lambda(\delta)}^{m}\norm{fw}_{L^{r}}
    \end{equation*}
    for every $f\in L^r_w(\mathbb R^n)$.
 \end{thm}

    \begin{thm}\label{limitcaseH}Let $m\in \mathbb N$, $0< \delta<\min\{1,(n-\alpha)/m\}$ and $r= n/((m-1)\delta+\alpha)$. Let $w$ be a weight such that $w^\beta\in A_{r/\beta,\infty}$ for some $1<\beta<r$. Let $T_\alpha$ be a fractional integral operator with kernel $K_\alpha\in S_{\alpha} \cap H_{\alpha,\Phi,m}(\delta)$ where $\Phi$ is a Young function verifying $\Phi^{-1}(t)\lesssim t^{\frac{\beta-1}{r}}$ for every $t>0$, and $\widetilde{\Phi}\in \mathcal B_{m\delta+\alpha,\beta}$. If $b\in \Lambda(\delta)$,
        the following statements are equivalent, \vspace{-0,3cm}
        \begin{enumerate}[label=(\roman*), leftmargin=0cm,itemindent=.6cm,labelwidth=\itemindent,labelsep=0cm,align=left]
            \item \label{ecu1H} $T_{\alpha,b}^{m}:L^{r}_w(\mathbb R^n)\hookrightarrow \mathbb L_w(\delta)$;
            \item There exists a positive constant $C$ such that
            \begin{equation}\label{ecu2H}
            \frac{\|w\chi_B\|_\infty}{|B|^{1+\frac{\delta}{n}}}
            \int_B\abs{\sum_{k=0}^{m}c_{k}\left[S(x,u,k)-(S(\cdot,u,k))_{B}\right]}dx\leq C \norm{fw}_{r},
            \end{equation}
            for every ball $B\subset \mathbb{R}^{n}$, $x, u\in B$ and $f\in L^{1}_{{\rm loc}}(\mathbb
            R^n)$.
        \end{enumerate}
    \end{thm}

\section{Auxiliary results}\label{Aux}

In this section we give some previous results. We begin with some
inequalities involving functions in $\Lambda(\delta)$.

\begin{lem}\label{puntualLip}Let $0<\delta < 1$ and $B\subset \mathbb{R}^{n} $ a ball. If $b\in \Lambda(\delta)$,
then
\begin{enumerate}[label=(\roman*), leftmargin=0cm,itemindent=.6cm,labelwidth=\itemindent,labelsep=0cm,align=left]
%   \item \label{puntopunto}for every $x,y\in \mathbb R^n$,
%   \[|b(x)-b(y)|\lesssim \|b\|_{\Lambda(\delta)} |x-y|^{\delta}.\]
%   Then, in particular, given a ball $B$,
%   \[\sup\limits_{x,y\in B}|b(x)-b(y)|\lesssim \|b\|_{\Lambda(\delta)} |B|^{\frac{\delta}{n}}. \]
    \item \label{puntoprom}for every $y\in \lambda B$, $\lambda\geq 1$,
        \[|b(y)-b_B|\leq C\|b\|_{\Lambda(\delta)}|\lambda B|^{\frac{\delta}{n}}.\]
    \item \label{promprom}for every $j\in \mathbb N$
    \begin{equation*}
    |b_{2^{j+1} B}-b_{2B}|\leq 2^n j|2^{j+1} B|^{\frac{\delta}{n}}\|b\|_{\Lambda(\delta)}.
    \end{equation*}
\end{enumerate}
\end{lem}

The following lemma is an easy consequence of condition
$S_{\alpha}$.

\begin{lem}\label{Salfabola}Let $K_\alpha$ be a kernel verifying condition $S_{\alpha}$ with $0<\alpha<n$. Then, for any ball $B=B(x_B,r_B)$, we have
    \[\int_{B} |K_\alpha(x-x_B)| dx\lesssim r_B^{\alpha}.\]
\end{lem}

\begin{proof}By changing variables first, we then split the integral into dyadic sets and use $S_\alpha$ condition in each set as it follows
    \begin{align*}
    \int_{B} |K_\alpha(x-x_B)| dx&=\int_{B(0,r_B)} |K_\alpha(z)| dz=\sum\limits_{j=1}^\infty \int_{2^{-j}r_B\leq |z|< 2(2^{-j} r_B)}|K_\alpha(z)| dz\\
    &\lesssim \sum\limits_{j=1}^\infty(2^{-j} r_B)^{\alpha}=r_B^{\alpha} \sum\limits_{j=1}^\infty 2^{-j\alpha}\thickapprox r_B^\alpha,
    \end{align*}
since $\alpha>0$.
\end{proof}

In order to obtain the boundedness result between Lebesgue spaces,
we prove the following key estimate, which shows how can we control
the higher order commutators of $T_\alpha$ by a fractional maximal
function via the sharp maximal operator $M^\sharp_{0,\gamma}$,
$0<\gamma<1$, given by
\[M^\sharp_0 f(x)=\sup\limits_{B\ni x}\inf\limits_{a\in\mathbb{R}} \frac{1}{|B|}\int_{B} |f(y)-a|\, dy,\]
where $M^\sharp_{0,\gamma} f:=M^\sharp_0(|f|^\gamma)^{1/\gamma}$.

\begin{lem}\label{pointwise}Let $m\in \mathbb{N}$, $0<\gamma<1/m$, $0<\alpha<n$ and $0< \delta
<\min\{1,(n-\alpha)/m\}$. Let $b\in \Lambda(\delta)$ and $T_\alpha$ a
fractional integral operator with kernel $K_\alpha\in S_\alpha$.
Then, there exists a positive constant $C$ such that
\begin{enumerate}[label=(\roman*), leftmargin=0cm,itemindent=.6cm,labelwidth=\itemindent,labelsep=0cm,align=left]
\item\label{nucleosPR} if $K_\alpha \in H_{\alpha,\infty}^*$,
\[M^\sharp_{0,\gamma}(T_{\alpha,b}^m f)(x)\lesssim \|b\|_{\Lambda(\delta)}^m
\left(\sum\limits_{j=0}^{m-1} M_{\theta_j,\gamma} (|T_{\alpha, b}^j f|)(x)+M_{\theta_0+\alpha}f(x)\right),\]
where $\theta_j=\delta(m-j)$, $j=0, \dots, m$.
\item\label{Hormanders} if $K_\alpha\in H_{\alpha,\Phi}$ for some Young function $\Phi$,
\[M^\sharp_{0,\gamma}(T_{\alpha,b}^m f)(x)\lesssim \|b\|_{\Lambda(\delta)}^m\left(\sum\limits_{j=0}^{m-1}
 M_{\theta_j,\gamma} (|T_{\alpha, b}^j f|)(x)+M_{\theta_0+\alpha,\widetilde{\Phi}}f(x)\right),\]
where $\theta_j=\delta(m-j)$, $j=0, \dots, m$, and
$\widetilde{\Phi}$ is the complementary function of $\Phi$.
\end{enumerate}
\end{lem}

\begin{rem}For $0< \delta < 1$, $m=1$ and $K_\alpha\in H_{\alpha,\infty}^*$ and homogeneous of degree $\alpha-n$, the proof of
\ref{nucleosPR} can be found in \cite{PR2} for a larger class of
Lipschitz spaces with variable parameter.
    \end{rem}

    \begin{proof}[Proof of Lemma \ref{pointwise}:] Fix $B$ a ball containing $x$, and decompose the commutator in the following way (see, for instance, \cite{GCHST} or \cite{P})
        \begin{equation}\label{reductionorder}T_{\alpha,b}^m f(x)=\sum\limits_{j=0}^{m-1}C_{j,m}(b(x)-b_{2B})^{m-j}T_{\alpha,b}^jf(x)+T_\alpha((b-b_{2B})^m f)(x).
        \end{equation}

        If we split $f=f_1+f_2$ where $f_2=f\chi_{2B}$, it is sufficient to estimate, for $0<\gamma <1/m$, the average
        \begin{equation}\label{promsharp}
        \left(\frac{1}{|B|}\int_B |T_{\alpha,b}^m f(y)-T_\alpha((b-b_{2B})^m f_2)(x_B)|^\gamma dy\right)^{1/\gamma}\leq I+II+III,
        \end{equation}
        where $x_B$ denotes the center of $B$, and
        \begin{align*}
        I&=\sum\limits_{j=0}^{m-1}\left(\frac{1}{|B|}\int_B (b(y)-b_{2B})^{(m-j)\gamma}|T_{\alpha,b}^j f(y)|^\gamma dy\right)^{\frac{1}{\gamma}},\\
        II&=\left(\frac{1}{|B|}\int_B |T_\alpha((b-b_{2B})^m f_1) (y)|^\gamma dy\right)^{\frac{1}{\gamma}},\\
        III&=\left(\frac{1}{|B|}\int_B |T_\alpha((b-b_{2B})^m f_2) (y)-T_\alpha((b-b_{2B})^m f_2)(x_B)|^\gamma dy\right)^{\frac{1}{\gamma}}.
        \end{align*}

        For simplicity, we will assume $\|b\|_{\Lambda(\delta)}=1$. We shall first estimate $I$. From Lemma \ref{puntualLip} \ref{puntoprom} we have
        \begin{align*}
        I&\lesssim\sum\limits_{j=0}^{m-1}\|b\|^{m-j}_{\Lambda(\delta)}|B|^{\frac{\delta (m-j)}{n}}\left(\frac{1}{|B|}\int_B |T_{\alpha,b}^j f(y)|^\gamma dy\right)^{\frac{1}{\gamma}}\\
        &=C \sum_{j=0}^{m-1}
        \left(\frac{1}{|B|^{1-\frac{(m-j)\delta \gamma}{n}}}\int_{B}\abs{T_{\alpha,b}^j f(y)}^{\gamma}\right)^{1/\gamma} \\
        &\lesssim  \sum_{j=0}^{m-1}
        \mathcal{M}_{\theta_{j},\gamma}(|T_{\alpha,b}^jf|)(x)
        \end{align*}
        where $\theta_{j}=(m-j)\delta $. Note that the last maximal operator is of fractional-type since $0<\theta_j<(m-j)(n-\alpha)/m\leq n$ for every $0\leq j\leq m-1$.

            We will now estimate $II$.
% Since, as we have mentioned, $K_\alpha\in S_{\alpha}$ even in the case that $K_\alpha$ is homogeneous of degree $\alpha-n$, we can proceed as in the proof of \cite[Theorem~3.6]{BLR}. If we denote by $R$ the radius of the ball $B$,
If $y\in B$ and $z\in 2B$, then $|y-z|< 3R$ and we have, by Lemma \ref{Salfabola}, that
            %       Since $T_\alpha$ is of weak-type $(1,n/(n-\alpha))$, by Kolmogorov's inequality, H\"older's inequality with $\Phi_m(t)=t\log(e+t)^m$ and its complementary function $\widetilde{\Phi}_{m}(t)=e^{t^{1/m}}-1$, and Lemma \ref{expLip} we can write
            \begin{align*}
            II&\leq \frac{1}{|B|}\int_B \left(\int_{2B}|K_{\alpha}(y-z)(b(z)-b_{2B})^m f(z)|dz\right)dy\\
            &\leq \frac{1}{|B|}\int_{2B}|b(z)-b_{2B}|^m |f(z)|\left( \int_{B(z, 3R)}|K_{\alpha}(y-z)|dy\right)dz\\
            &\lesssim \frac{|3B|^\frac{\alpha}{n}}{|B|}\int_{2B} |b(z)-b_{2B}|^m |f (z)| dy
            \end{align*}

            From Lemma \ref{puntualLip} \ref{puntoprom}, we can estimate $|b(z)-b_{2B}|^m$ by $\|b\|_{\Lambda(\delta)}^m |2B|^{\frac{m\delta}{n}}=\|b\|_{\Lambda(\delta)}^m |2B|^{\frac{\theta_0}{n}}$ to obtain
            \begin{align*}
            II&\lesssim \|b\|_{\Lambda(\delta)}^m  \frac{1}{|2B|^{1-\frac{\theta_0+\alpha}{n}}}\int_{2B}  |f (z)| dz\leq \|b\|_{\Lambda(\delta)}^m M_{\theta_0+\alpha}f(x).
            \end{align*}
            Since $0<\delta<(n-\alpha)/m$, it is clear that $0<\theta_0+\alpha<n$, so $ M_{\theta_0+\alpha}$ is a fractional-type maximal operator.

%           If $\delta=0$, we can apply H\"older's inequality with $C_{m,0}(t)=e^{t^{1/m}}-1$ and $\widetilde{C_{m,0}}(t)\thickapprox t\log(e+t)^m$ and Lemma \ref{expLip} to get
%           \begin{align*}
%           II&\lesssim |B|^{\frac{\alpha}{n}}\norm{(b-b_{2B})^m}_{(\exp L)^{1/m}, 2B}\norm{f}_{L (\log L)^{m},2B}\\
%           &\lesssim  |B|^{\frac{\alpha}{n}}\norm{b-b_{2B}}_{\exp L, 2B}^m \norm{f}_{L (\log L)^{m},2B}\\
%           &\lesssim  \norm{b}_{\mathbb{L}(0)}^{m} |2B|^{\frac{\alpha}{n}}\norm{f}_{L (\log L)^{m},2B}\\
%           &\lesssim  M_{\theta_0+\alpha,L(\log L)^{m}}f(x),
%           \end{align*}
%           since $\theta_0 =\delta m=0$.

        In order to estimate $III$, we first observe that, by Jensen's inequality
        \[III\leq \frac{1}{|B|}\int_B |T_{\alpha,b}^m\left((b-b_{2B})^m f_2\right)(y)-T_\alpha((b-b_{2B})^m f_2)(x_B)|dy,\]
and, setting $B_{j}=2^{j}B$, the integrand can be estimated, using Lemma \ref{puntualLip} \ref{puntoprom}, as follows
        \begin{align}\label{III}
        |T_{\alpha,b}^m\left((b-b_{2B})^m f_2\right)(y)&-T_\alpha((b-b_{2B})^m f_2)(x_B)|\\
        \nonumber&\leq \sum\limits_{j=1}^\infty \int_{B_{j+1}\setminus B_j}|K_\alpha(y-z)-K_\alpha(x_B-z)||b(z)-b_{2B}|^m|f(z)|dz\\
        \nonumber&\lesssim \|b\|_{\Lambda(\delta)}^{m} \sum\limits_{j=1}^\infty |B_{j+1}|^{\frac{m\delta}{n}}\int_{B_{j+1}\setminus B_j}|K_\alpha(y-z)-K_\alpha(x_B-z)| |f(z)|dz.
\end{align}
%
%Thus, from Lemma \ref{puntualLip} \ref{puntoprom} and \ref{promprom} we have
%\begin{align}\label{III}
%    |T_{\alpha,b}^m\left((b-b_{2B})^m f_2\right)&(y)-T_\alpha((b-b_{2B})^m f_2)(x_B)|\\
%    &\nonumber \lesssim \|b\|_{\Lambda(\delta)}^m \sum\limits_{j=1}^\infty |B_{j+1}|^{\frac{m\delta}{n}}\int_{B_{j+1}\setminus B_j}|K_\alpha(y-z)-K_\alpha(x_B-z)||f(z)|dz\\
%    &\nonumber \quad + \|b\|_{\Lambda(\delta)}^m \sum\limits_{j=1}^\infty j^m|B_{j+1}|^{\frac{m\delta}{n}}\int_{B_{j+1}\setminus B_j}|K_\alpha(y-z)-K_\alpha(x_B-z)| |f(z)|dz\\
%    &\nonumber \lesssim \|b\|_{\Lambda(\delta)}^m \sum\limits_{j=1}^\infty j^m|B_{j+1}|^{\frac{m\delta}{n}}\int_{B_{j+1}\setminus B_j}|K_\alpha(y-z)-K_\alpha(x_B-z)| |f(z)|dz.
%\end{align}

Here, we must distinguish the cases $K_\alpha\in H_{\alpha,\infty}^*$ and $K_\alpha\in H_{\alpha, \Phi}$.

If $K_\alpha\in H_{\alpha,\infty}^*$,
\begin{align*}
    |T_{\alpha,b}^m((b-b_{2B})^m f_2)(y)-T_\alpha&((b-b_{2B})^m f_2)(x_B)|\\
        &\lesssim \|b\|_{\Lambda(\delta)}^m \sum\limits_{j=1}^\infty|B_{j+1}|^{\frac{\delta m}{n}}\int_{B_{j+1}\setminus B_j}\frac{|y-x_B|^\eta}{|y-z|^{n-\alpha+\eta}} |f(z)|dz\\
        &\lesssim\|b\|_{\Lambda(\delta)}^m\sum_{j=1}^{\infty}|B_{j+1}|^{\frac{m\delta+\alpha}{n}} 2^{-j\eta} \frac{1}{|B_{j+1}|}\int_{B_{j+1}}|f(z)|dz\\
        &\thickapprox \|b\|_{\Lambda(\delta)}^m\sum_{j=1}^{\infty} 2^{-j\eta}\frac{1}{|B_{j+1}|^{1-\frac{\theta_0+\alpha}{n}}}\int_{B_{j+1}}|f(y)|dy\\
        &\leq \norm{b}_{\Lambda(\delta)}^{m}M_{\theta_0+\alpha}f(x)\sum_{j=1}^{\infty}2^{-j\eta}\lesssim \norm{b}_{\Lambda(\delta)}^{m}M_{\theta_0+\alpha}f(x),
        \end{align*}
since $\eta>0$. Therefore
\[III\lesssim \|b\|_{\Lambda(\delta)}^m M_{\theta_0+\alpha}f(x).\]

Let us now consider the case $K_\alpha\in H_{\alpha, \Phi}$. Applying H\"older's inequality with $\Phi$ and $\widetilde{\Phi}$ in \eqref{III}, we obtain
\begin{align*}
|T_{\alpha,b}^m&((b-b_{2B})^m f_2)(y)-T_\alpha((b-b_{2B})^m f_2)(x_B)|\\
&\lesssim \|b\|_{\Lambda(\delta)}^m \sum\limits_{j=1}^\infty |B_{j+1}|^{\frac{m\delta}{n}+1} \|\left(K_\alpha(\cdot-(y-x_B))-K_\alpha(\cdot)\right)\chi_{|\cdot|\sim 2^j R}\|_{\Phi,B_{j+1}}  \|f\|_{\widetilde{\Phi},B_{j+1}}\\
&\lesssim \|b\|_{\Lambda(\delta)}^m \sum\limits_{j=1}^\infty |B_{j+1}|^{1-\frac{\alpha}{n}} \|\left(K_\alpha(\cdot-(y-x_B))-K_\alpha(\cdot)\right)\chi_{|\cdot|\sim 2^j R}\|_{\Phi,B_{j+1}}  |B_{j+1}|^{\frac{m\delta+\alpha}{n}}\|f\|_{\widetilde{\Phi},B_{j+1}}\\
&\lesssim \|b\|_{\Lambda(\delta)}^m  M_{\theta_0+\alpha,\widetilde{\Phi}}f(x) \sum\limits_{j=1}^\infty(2^{j}R)^{n-\alpha} \|\left(K_\alpha(\cdot-(y-x_B))-K_\alpha(\cdot)\right)\chi_{|\cdot|\sim 2^j R}\|_{\Phi,B_{j+1}}\\
&\lesssim  \|b\|_{\Lambda(\delta)}^m M_{\theta_0+\alpha,\widetilde{\Phi}}f(x).
\end{align*}
Therefore,
\[III\lesssim \|b\|_{\Lambda(\delta)}^m M_{\theta_0+\alpha,\widetilde{\Phi}}f(x).\]

Combining all these estimates, we obtain the desired pointwise inequalities.
\end{proof}

The following result is a variant of the well-known Fefferman-Stein's inequality (see \cite{GCRF}) and it will be a key estimate to prove Theorem \ref{teo1}.

\begin{lem}[\cite{P2}]Let $0<p<\infty$ and $0<\gamma<1$. Let $w$ be a weight in the $A_\infty$ class. Then, there exists a positive constant $C$ such that
\begin{equation}\label{FeffSt}
\int_{\mathbb R^n} M_\gamma f(x)^p w(x) dx\leq C[w]_{A_\infty} \int_{\mathbb R^n} M_{0,\gamma}^\sharp f(x)^p w(x) dx
\end{equation}
for every measurable function $f$.
\end{lem}

We shall also need two results involving the boundedness of fractional maximal operators associated with Young functions, that can be found in  \cite{BDP}.

\begin{thm}[\cite{BDP}]\label{BDPcteH} Let $0<\alpha<n$, $1\leq \beta<p<n/\alpha$ and $1/q=1/p-\alpha/n$. Let $w$ be a weight such that $w^\beta\in A_{p/\beta,q/\beta}$. Let $\Psi$ be a Young function that satisfies $\Psi\in\mathcal B_{\alpha,\beta}$. Then, $M_{\alpha,\Psi}$ is bounded from $L^{p}(w^p,\mathbb{R}^n)$ into $L^{q}(w^q,\mathbb{R}^n)$.
\end{thm}

Note that if $\Psi=t^\beta(1+\log^{+}t)^{\gamma}$ for any $\gamma\geq 0$, then $\Psi\in \mathcal{B}_{\alpha,\beta}$ and the following result holds.

%is a Young function of $L\log L$ type, the above result can be improved. Nevertheless, the sufficiency is a particular case of Theorem \ref{BDPcteH}.

\begin{thm}[\cite{BDP}]\label{BDPcteLlogL} Let $0<\alpha<n$, $1<p<n/\alpha$ and $1/q=1/p-\alpha/n$. Let $w$ be a weight and  $\Psi(t)=t^\beta(1+\log^+t)^\gamma$ where $1\leq \beta<p$ and $\gamma\geq 0$. Then, $M_{\alpha,\Psi}$ is bounded from $L^{p}(w^p,\mathbb{R}^n)$ into $L^{q}(w^q,\mathbb{R}^n)$ if and only if $w^\beta\in A_{p/\beta,q/\beta}$.
\end{thm}

In order to prove Theorem \ref{limitcase}, we shall need the following estimate.

\begin{lem}\label{lema m>1}Let $0< \delta<\min\{\eta,(n-\alpha)/(m-1)\}$, for $0<\eta\leq 1$. Let  $r=n/((m-1)\delta+\alpha)$, $w\in A_{r,\infty}$, $b\in \Lambda(\delta)$ and $f\in L^r_w(\mathbb R^n)$. Let $B\subset \mathbb{R}^{n}$ be a ball and $f_2=f\chi_{\mathbb R^n\setminus 2B}$. If $T_\alpha$ is a fractional integral operator with kernel $K_\alpha\in S_\alpha\cap H_{\alpha,\infty}^*$, then, for every $x,u\in B$,
    \begin{equation*}
    \abs{T_{\alpha}((b-b_B)^{k}f_{2})(x)-T_{\alpha}((b-b_B)^{k}f_{2})(u)}\lesssim\frac{\norm{b}_{\Lambda(\delta)}^{k}\norm{fw}_{r}
    |B|^{\frac{\delta(k-m+1)}{n}}}{\norm{w\chi_B}_{\infty}}
    \end{equation*}
for each $k=0,...,m$.
\end{lem}

\begin{proof}[Proof of Lemma \ref{lema m>1}] If $K_\alpha\in S_\alpha\cap H_{\alpha,\infty}^*$, by taking $x,u\in B$, and $0\leq k\leq m$, and setting $B_j=2^j B$, we have from Lemma \ref{puntualLip} \ref{puntoprom} that
\begin{align}\label{lemmalimite}
  |T_{\alpha}((b-b_{B})^{k}f_{2})(x)-T_{\alpha}&((b-b_{B})^{k}f_{2})(u)|\\
  \nonumber &\leq \int_{(2B)^{c}}\abs{K_{\alpha}(x-y)-K_{\alpha}(u-y)}\abs{b(y)-b_{B}}^{k}\abs{f(y)}dy\\
   \nonumber &\lesssim \|b\|_{\Lambda(\delta)}^k \sum_{j=1}^{\infty}|B_{j+1}|^{\frac{\delta k}{n}}\int_{B_{j+1}\setminus B_j}\frac{\abs{x-u}^{\eta}}{\abs{y-u}^{n+\eta-\alpha}}\abs{f(y)}dy\\
   \nonumber&\lesssim \norm{b}_{\Lambda(\delta)}^{k}\sum_{j=1}^{\infty}
   \frac{\abs{B_{j+1}}^{\frac{\delta k}{n}}|B|^{\frac{\eta}{n}}}{\abs{B_{j+1}}^{\frac{n+\eta -\alpha}{n}}}\int_{B_{j+1}\setminus B_j}\abs{f(y)}dy.
  \end{align}

Now by H\"older's inequality and the fact that $w\in A_{r,\infty}$ with $r=n/(\alpha+(m-1)\delta)$, we get
\begin{align*}
 |T_{\alpha}((b-b_{B})^{k}f_{2})(x)-T_{\alpha}&((b-b_{B})^{k}f_{2})(u)|\\
& \lesssim\norm{b}_{\Lambda(\delta)}^{k}\norm{fw}_{r}
\sum_{j=1}^{\infty}
\frac{\abs{B_{j+1}}^{\frac{\delta k}{n}}}{\abs{B_{j+1}}^{1-\frac{\alpha}{n}}2^{j\eta}}
\norm{w^{-1}\chi_{B_{j+1}}}_{r'}\\
& \lesssim\norm{b}_{\Lambda(\delta)}^{k}\norm{fw}_{r}
\norm{w\chi_{B}}_{\infty}^{-1}
\sum_{j=1}^{\infty}
\frac{\abs{B_{j+1}}^{\frac{\delta k}{n}-1+\frac{\alpha}{n}+1-\frac{\alpha+(m-1)\delta}{n}}}{2^{j\eta}}\\
& \lesssim\norm{b}_{\Lambda(\delta)}^{k}\norm{fw}_{r}
\norm{w\chi_{B}}_{\infty}^{-1}|B|^{\frac{\delta(k-m+1)}{n}}
\sum_{j=1}^{\infty}2^{j(\delta(k-m+1)-\eta)}\\
& \lesssim\norm{b}_{\Lambda(\delta)}^{k}\norm{fw}_{r}
\norm{w\chi_{B}}_{\infty}^{-1}|B|^{\frac{\delta(k-m+1)}{n}}
\end{align*}
where the series is summable since $0\leq k\leq m$ and $\delta<\eta$.
%
%If  we now assume that $K_\alpha\in S_\alpha\cap H_{\alpha,\Phi,m}$, we can rewrite \eqref{lemmalimite} as
%\begin{align*}
%|T_{\alpha}&((b-b_{B})^{k}f_{2})(x)-T_{\alpha}((b-b_{B})^{k}f_{2})(u)|\\
%&\lesssim \|b\|_{\Lambda(\delta)}^k \sum_{j=1}^{\infty}|B_{j+1}|^{\frac{\delta k}{n}}\int_{B_{j+1}\setminus B_j}|K_\alpha(x-y)-K_\alpha(u-y)|\abs{f(y)}dy.
%\end{align*}
%Me parece que acá vamos a tener que pedir directamente $\Phi(t)=t^q$ pero el $q$ es el que sale de la apertura así que no queda lindo pedirlo en el enunciado...
\end{proof}

 \begin{lem}\label{lemalimite}Let $m\in \mathbb N$, $0<\alpha<n$, $0<\delta<\min\{1, (n-\alpha)/(m-1)\}$, $r=n/((m-1)\delta+\alpha)$, $b\in \Lambda(\delta)$ and $f\in L^{r}_w(\mathbb R^n)$ where $w$ is a weight such that $w^\beta\in A_{r/\beta,\infty}$ for some $1<\beta<r$. Let $B\subset \mathbb{R}^{n}$ be a ball and $f_2=f\chi_{\mathbb R^n\setminus 2B}$. If $T_\alpha$ is a fractional integral operator with kernel $K_\alpha\in  H_{\alpha,\Phi,m}(\delta)$,  where $\Phi$ is a Young function verifying $\Phi^{-1}(t)\lesssim t^{\frac{\beta-1}{r}}$ for every $t>0$, then, for every $x,u\in B$,
        \begin{equation*}
        \abs{T_{\alpha}((b-b_B)^{k}f_{2})(x)-T_{\alpha}((b-b_B)^{k}f_{2})(u)}
        \lesssim\frac{\norm{fw}_{r} \norm{b}_{\Lambda(\delta)}^{k}\abs{B}^{\frac{\delta (k-m+1)}{n}}}{\|w\chi_B\|_\infty}
        \end{equation*}
        for each $k=0,...,m$.
    \end{lem}

    \begin{proof}
        Fix $x,u\in B$ and $0\leq k\leq m$. Setting $B_u=B(u,R)$ which satisfies $B\subset 2B_u\subset 4B$, and using Lemma \ref{puntualLip} \ref{puntoprom} we have
        \begin{align*}
        |T_{\alpha}((b-b_{B})^{k}f_{2})(x)-&T_{\alpha}((b-b_{B})^{k}f_{2})(u)|\\
        &\lesssim \int_{\mathbb R^n\setminus B_u} |b(y)-b_{B_u}|^k|K_\alpha(x-y)-K_\alpha(u-y)| |f(y)|dy\\
        &\lesssim \sum\limits_{j=0}^\infty \int_{2^{j+1}B_u\setminus 2^j B_u} |b(y)-b_{B_u}|^k|K_\alpha(x-y)-K_\alpha(u-y)| |f(y)|dy\\
        &\lesssim \|b\|_{\Lambda(\delta)}^k\sum\limits_{j=0}^\infty |2^{j+1}B_u|^{\frac{\delta k}{n}}\int_{2^{j+1}B_u\setminus 2^j B_u} |K_\alpha(x-y)-K_\alpha(u-y)| |f(y)|dy.
        \end{align*}

    Since $1/r+1/(r/\beta)'=1-(\beta-1)/r$, we can use H\"older's inequality with $\Phi^{-1}(t)t^{1/r}t^{1/(r/\beta)'}\lesssim t$ and the fact that $w^\beta\in A_{r/\beta,\infty}$, to get
        \begin{align}\label{lemma3.9}
        |T_{\alpha}&((b-b_{B})^{k}f_{2})(x)-T_{\alpha}((b-b_{B})^{k}f_{2})(u)|\\
        \nonumber&\lesssim \|b\|_{\Lambda(\delta)}^k\|fw\|_{r}\sum\limits_{j=0}^\infty |2^{j+1}B_u|^{\frac{\delta k}{n}+1-\frac{1}{r}} \||K_\alpha(\cdot-(u-x))-K_\alpha(\cdot)|\chi_{|\cdot|\sim 2^{j}R}\|_{\Phi,2^{j+1}B_u}\frac{\|w^{-1}\chi_{2^{j+1}B_u}\|_{(r/\beta)'}}{|2^{j+1}B_u|^{1/(r/\beta)'}}\\
        \nonumber&\lesssim \frac{\|b\|_{\Lambda(\delta)}^k |B|^{\frac{\delta k+\alpha}{n}-\frac{1}{r}}\|fw\|_r }{\|w\chi_{2B_u}\|_\infty}\sum\limits_{j=1}^\infty   (2^jR)^{n-\alpha} 2^{j(\delta k+\alpha-\frac{n}{r})} \||K_\alpha(\cdot-(u-x))-K_\alpha(\cdot)|\chi_{|\cdot|\sim 2^{j}R}\|_{\Phi,2^{j+1}B_u}\\
        \nonumber&\leq \frac{\|b\|_{\Lambda(\delta)}^k |B|^{\frac{\delta (k-m+1)}{n}}\|fw\|_r }{\|w\chi_{B}\|_\infty} \sum\limits_{j=1}^\infty 2^{jm\delta}(2^jR)^{n-\alpha} \||K_\alpha(\cdot-(u-x))-K_\alpha(\cdot)|\chi_{|\cdot|\sim 2^{j}R}\|_{\Phi,2^{j+1}B_u}\\
        \nonumber&\lesssim \frac{\|b\|_{\Lambda(\delta)}^k |B|^{\frac{\delta (k-m+1)}{n}}\|fw\|_r }{\|w\chi_{B}\|_\infty},
        \end{align}
        where we have used that $\delta k+\alpha-n/r\leq m\delta$ for $m\in \mathbb N$, and that $K_\alpha\in H_{\alpha, \Phi,m}(\delta)$.
\end{proof}

\section{Proofs of main results}\label{proofs}

\begin{proof}[Proof of Theorem \ref{teo1}:] The proof will be done by induction {and, without loss of generality, we shall assume $\|b\|_{\Lambda(\delta)}=1$}. Notice that when $m=0$, $1/q=1/p-\alpha/n$ and the boundedness result is known to be true for $A_{p,q}$ weights (see \cite{BDP} in the more general setting of variable Lebesgue spaces).

    Fix $m\in \mathbb N$ and define the following auxiliary exponents
    \begin{equation*}
    \frac{1}{p_{j}}=\frac{1}{q}+\frac{\delta(m-j)}{n}=\frac{1}{p}-\frac{j\delta+\alpha}{n}, \quad j=0,\dots, m.
    \end{equation*}
    Clearly, $p_m=q$ and, if $\theta_j=(m-j)\delta$, we have that
    \begin{equation}\label{p_j}
    \frac{1}{p_{j}}=\frac{1}{q}+\frac{\theta_j}{n}=\frac{1}{p}-\frac{j\delta+\alpha}{n}, \quad j=0,\dots, m.
    \end{equation}
    Notice also that $p\leq p_j\leq p_l\leq q$ for every $0\leq j\leq l\leq m$.

    It is easy to see that $w\in A_{p,q}$ yields $w^\gamma\in A_{\frac{p}{\gamma},\frac{q}{\gamma}}$ for every $0<\gamma<1$. Moreover, from the properties of these classes, we have that $w^\gamma\in A_{\frac{p_j}{\gamma},\frac{p_l}{\gamma}}$ for every $0\leq j\leq l\leq m$.

    By applying Fefferman-Stein's inequality \eqref{FeffSt} with $w^q\in A_{1+q/p'}\subset A_\infty$, we get
    \[\|wT_{\alpha,b}^m f\|_{q}\leq \|w M_\gamma(T_{\alpha,b}^m f)\|_{q} \lesssim \|w M_{0,\gamma}^\sharp (T_{\alpha,b}^m f)\|_{q}.\]
    Now, since $K_\alpha\in S_\alpha\cap H_{\alpha,\infty}^*$, from Lemma \ref{pointwise} we have that
    {\begin{equation}\label{teo1maximales}
    \|wT_{\alpha,b}^m f\|_{q}   \lesssim \sum\limits_{j=0}^{m-1} \|wM_{\theta_j,\gamma} (|T_{\alpha,b}^j f|)\|_{q}+\|wM_{\theta_0+\alpha}f\|_{q}.
    \end{equation}}

    Since $w\in A_{p,q}$ and $1/q=1/p-(\theta_0+\alpha)/n$, we have that
    {\[\|wT_{\alpha,b}^m f\|_{q}\lesssim \sum\limits_{j=0}^{m-1} \|wM_{\theta_j,\gamma} (|T_{\alpha,b}^j f|)\|_{q}+\|fw\|_{p}.\]}

On the other hand, since $w^{\gamma} \in A_{\frac{p_{j}}{\gamma},\frac{q}{\gamma}}$ for every $j=1,\dots, m-1$, then the fractional maximal operator $M_{\theta_j\gamma}$ is bounded from $L^{\frac{p_{j}}{\gamma}}(\mathbb R^n)$ to $L^{\frac{q}{\gamma}}(\mathbb R^n)$. Thus, we have that
    {\begin{align*}
    \|wT_{\alpha,b}^m f\|_{q}&\lesssim \sum\limits_{j=0}^{m-1} \|w^\gamma M_{\theta_j\gamma} (|T_{\alpha,b}^j f|^\gamma)\|_{q/\gamma}^{1/\gamma}+\|wf\|_{p}\\
    &\lesssim \sum\limits_{j=0}^{m-1} \|w^\gamma (T_{\alpha,b}^j f)^{\gamma}\|_{p_j/\gamma}^{1/\gamma}+\|wf\|_{p}\\
    &\lesssim \sum\limits_{j=0}^{m-1} \|w T_{\alpha,b}^j f\|_{p_{j}}+\|wf\|_{p}.
    \end{align*}}

    Since $1/p_{j}=1/p-(j\delta+\alpha)/n$ and $w\in A_{p,p_j}$, {and recalling that $\|b\|_{\Lambda(\delta)}=1$,} we apply the inductive hypothesis to get
    {\begin{equation*}
    \|wT_{\alpha,b}^m f\|_{q}\lesssim \sum\limits_{j=0}^{m-1} \|w f\|_{p}+\|wf\|_{p}\lesssim \|wf\|_{p}. \qedhere
    \end{equation*}}
\end{proof}

\begin{proof}[Proof of Theorem \ref{boundednessLipw}:] Fix $f\in L^{r}_w(\mathbb R^n)$. For a ball $B\subset \mathbb{R}^{n}$, set $f_{1}=f\chi_{2B}$, $f_{2}=f-f_{1}$ and $a_{B}=\frac{1}{|B|}\int_BT_{\alpha,b}^{m}f_{2}$. Then,
    \begin{align*}
    \frac{\norm{w\chi_B}_{\infty}}{|B|}\int_B\abs{T_{\alpha,b}^{m}f(x)-a_{B}}dx&\leq \frac{\norm{w\chi_B}_{\infty}}{|B|}\int_B\abs{T_{\alpha,b}^{m}f_{1}(x)}dx+\frac{\norm{w\chi_B}_{\infty}}{|B|}\int_B\abs{T_{\alpha,b}^{m}f_{2}(x)-a_{B}}dx\\
    &= I_1+I_{2}.
    \end{align*}
    Let us first notice that, since $w\in A_{r,\infty}$, there exists $1<s'<r$ such that $w\in A_{s',\infty}$ and, we can choose $1<q<\infty$ such that $\frac{1}{q}+\frac{1}{r}+\frac{1}{s}=1$.

    For $I_1$ we write
    \begin{align*}
    I_1=\frac{\norm{w\chi_B}_{\infty}}{|B|}\int_B\abs{\int_{2B}(b(x)-b(y))^{m}K_\alpha(x-y)f(y)dy}dx.
    \end{align*}
    By using Tonelli's theorem and the fact that $b\in \Lambda(\delta)$, we obtain
    \begin{align*}
    I_1&\lesssim  \norm{w\chi_B}_{\infty} \int_{2B}|f(y)|\left(\frac{1}{(2R)^n}\int_{2B} |b(x)-b(y)|^{m} |K_\alpha(x-y)|dx\right)dy\\
    &\lesssim \|b\|_{\Lambda(\delta)}^m |2B|^{\frac{\delta m}{n}}\norm{w\chi_B}_{\infty} \int_{2B}|f(y)|  \left(\fint_{2B}  |K_\alpha(x-y)|dx \right)dy
    \end{align*}

    We notice that for $x,y\in 2B$, if $R$ is the radius of $B$, then $x\in B(y,4R)$ so we can use Lemma \ref{Salfabola} and H\"older's inequality to have
    \begin{align*}
    I_1&\lesssim \|b\|_{\Lambda(\delta)}^m |2B|^{\frac{m\delta+\alpha}{n}-1}\norm{w\chi_B}_{\infty} \|fw\|_{r} \|w^{-1}\chi_{2B}\|_{r'} \\
    &\leq \|b\|_{\Lambda(\delta)}^m \|fw\|_{r}|2B|^{\frac{m\delta+\alpha}{n}-\frac{1}{r}}\norm{w\chi_B}_{\infty}  \frac{\|w^{-1}\chi_{2B}\|_{r'}}{|2B|^{1/r'}} \\
    &\lesssim [w]_{A_{r,\infty}} \|b\|_{\Lambda(\delta)}^m \|fw\|_{r} |B|^{\frac{\tilde{\delta}}{n}}.
    \end{align*}

    For $I_2$, we first estimate the difference $\abs{T^{m}_{\alpha,b}f_{2}(x)-(T^{m}_{\alpha,b}f_{2})_B}$ for every $x\in B$. Since
\[ \abs{T^{m}_{\alpha,b}f_{2}(x)-(T^{m}_{\alpha,b}f_{2})_B}\leq \frac{1}{|B|}\int_B\abs{T^{m}_{\alpha,b}f_{2}(x)-T^{m}_{\alpha,b}f_{2}(y)}dy,\]
we analyze $A=\abs{T^{m}_{\alpha,b}f_{2}(x)-T^{m}_{\alpha,b}f_{2}(y)}$. If $x,y\in B$
\begin{align*}
A &\leq \int_{(2B)^c}\abs{(b(x)-b(z))^{m}K_\alpha(x-z)-(b(y)-b(z))^{m}K_\alpha(y-z)}\abs{f(z)}dz\\
&\leq \int_{(2B)^c}\abs{b(x)-b(z)}^{m}\abs{K_\alpha(x-z)-K_\alpha(y-z)}\abs{f(z)}dz\\
  &\quad+ \int_{(2B)^c}\abs{(b(x)-b(z))^{m}-(b(y)-b(z))^{m}}\abs{K_\alpha(y-z)}\abs{f(z)}dz\\
   &\leq \int_{(2B)^c}\abs{b(x)-b(z)}^{m}\abs{K_\alpha(x-z)-K_\alpha(y-z)}\abs{f(z)}dz\\
  &\quad+\abs{b(x)-b(y)}\sum_{k=0}^{m-1}\int_{(2B)^c}\abs{b(x)-b(z)}^{m-1-k}\abs{b(y)-b(z)}^{k}\abs{K_\alpha(y-z)}\abs{f(z)}dz\\
&=I_{3}+I_{4}.
\end{align*}

By the definition of $\Lambda(\delta)$, we get that
\begin{align*}
  I_{3}&\lesssim \norm{b}_{\Lambda(\delta)}^{m}\int_{(2B)^c}\abs{x-z}^{\delta m}\abs{K_{\alpha}(x-z)-K_{\alpha}(y-z)}\abs{f(z)}dz\\
&\lesssim \norm{b}_{\Lambda(\delta)}^{m}\sum_{j=1}^{\infty}\int_{2^{j+1}B\setminus 2^{j}B}\abs{x-z}^{\delta m}
\frac{\abs{x-y}^{\eta}}{\abs{x-z}^{n-\alpha+\eta}}\abs{f(z)}dz\\
&\lesssim \norm{b}_{\Lambda(\delta)}^{m}\sum_{j=1}^{\infty}
\frac{2^{j\delta m}|B|^{\frac{\delta m}{n}}}{2^{j(n-\alpha+\eta)}|B|^{1-\frac{\alpha}{n}}}
\int_{2^{j+1}B\setminus 2^{j}B}\abs{f(z)}dz.
\end{align*}
Then, by H\"{o}lder's inequality, the definition of $\tilde{\delta}$ and the fact that $w\in A_{r,\infty}$, we deduce that
\begin{equation*}
  I_{3}\lesssim \norm{b}_{\Lambda(\delta)}^{m}\norm{fw}_{r}|B|^{\frac{\tilde{\delta}}{n}}
  \sum_{j=1}^{\infty}2^{j(\tilde{\delta}-\eta)}\frac{\norm{w^{-1}\chi_{2^{j+1}B}}_{r'}}{|2^{j+1}B|^{1/r'}}\lesssim [w]_{A_{r,\infty}}\frac{\norm{b}_{\Lambda(\delta)}^{m}\norm{fw}_{r}|B|^{\frac{\tilde{\delta}}{n}}}{\norm{w\chi_B}_{\infty}}.
\end{equation*}

In order to estimate $I_{4}$, we use that $b\in \Lambda(\delta)$ and the smoothness condition $S_{\alpha}$ on the kernel to get that
\begin{align*}
  I_{4}&\lesssim \norm{b}_{\Lambda(\delta)}\abs{x-y}^{\delta}\norm{b}_{\Lambda(\delta)}^{m-1}\sum_{k=0}^{m-1}
  \sum_{j=1}^{\infty}\int_{2^{j+1}B\setminus 2^{j}B}\abs{x-z}^{\delta(m-1-k)}
 \abs{y-z}^{\delta k}\abs{K_{\alpha}(x-z)}\abs{f(z)}dz\\
  &\lesssim \norm{b}^{m}_{\Lambda(\delta)}|B|^{\frac{\delta}{n}}
  \sum_{j=1}^{\infty}|2^{j+1}B|^{\frac{\delta(m-1)}{n}}\int_{2^{j+1}B\setminus 2^{j}B}\abs{K_{\alpha}(x-z)}\abs{f(z)}dz.
\end{align*}
Then, by Tonelli's Theorem and the smoothness condition $S_\alpha$
\begin{align*}
|T^{m}_{\alpha,b}&f_{2}(x)-(T^{m}_{\alpha,b}f_{2})_B|\\
 &\lesssim \frac{1}{|B|} \int_B (I_3+I_4) \;dy\\
& \lesssim \frac{[w]_{A_{r,\infty}}}{\norm{w\chi_B}_{\infty}}\norm{b}_{\Lambda(\delta)}^{m}\norm{fw}_{r}|B|^{\frac{\tilde{\delta}}{n}}\\
&\quad+\frac{\norm{b}^{m}_{\Lambda(\delta)}|B|^{\frac{\delta}{n}}}{|B|}
\sum_{j=1}^{\infty}|2^{j+1}B|^{\frac{\delta(m-1)}{n}}\int_B\int_{2^{j+1}B\setminus 2^{j}B}\abs{K_{\alpha}(x-z)}\abs{f(z)}dz dy\\
&\lesssim \frac{[w]_{A_{r,\infty}}}{\norm{w\chi_B}_{\infty}}\norm{b}_{\Lambda(\delta)}^{m}\norm{fw}_{r}|B|^{\frac{\tilde{\delta}}{n}}\\
&\quad+\norm{b}^{m}_{\Lambda(\delta)}|B|^{\frac{\delta}{n}}
\sum_{j=1}^{\infty}|2^{j+1}B|^{\frac{\delta(m-1)}{n}}\int_{2^{j+1}B\setminus 2^{j}B}\abs{f(z)}\left(\fint_{B(z,2^{j+2}R)} |K_\alpha(y-z)| dy\right)dz \\
&\lesssim \frac{[w]_{A_{r,\infty}}}{\norm{w\chi_B}_{\infty}}\norm{b}_{\Lambda(\delta)}^{m}\norm{fw}_{r}|B|^{\frac{\tilde{\delta}}{n}}+\norm{b}^{m}_{\Lambda(\delta)}|B|^{\frac{\delta}{n}}
\sum_{j=1}^{\infty}|2^{j+1}B|^{\frac{(m-1)\delta+\alpha}{n}-1}\int_{2^{j+1}B\setminus 2^{j}B}\abs{f(z)}dz \\
&\lesssim \frac{[w]_{A_{r,\infty}}}{\norm{w\chi_B}_{\infty}}\norm{b}_{\Lambda(\delta)}^{m}\norm{fw}_{r}|B|^{\frac{\tilde{\delta}}{n}}\\
&\quad+\norm{b}^{m}_{\Lambda(\delta)}|B|^{\frac{m\delta+\alpha}{n}-\frac{1}{r}}
\sum_{j=1}^{\infty}2^{j\left(\frac{(m-1)\delta+\alpha}{n}-\frac{1}{r}\right)}\|fw\|_r \frac{\|w^{-1}\chi_{2^{j+1}B}\|_{r'}}{|2^{j+1}B|^{1/r'}}\\
&\lesssim \frac{[w]_{A_{r,\infty}}}{\norm{w\chi_B}_{\infty}}\norm{b}_{\Lambda(\delta)}^{m}\norm{fw}_{r}|B|^{\frac{\tilde{\delta}}{n}}\left(1+\sum_{j=1}^{\infty}2^{j(\tilde{\delta}-\delta)}\right)\\
&\lesssim \frac{[w]_{A_{r,\infty}}}{\norm{w\chi_B}_{\infty}}\norm{b}_{\Lambda(\delta)}^{m}\norm{fw}_{r}|B|^{\frac{\tilde{\delta}}{n}}
\end{align*}
since $\tilde{\delta}<\delta$. Therefore, we deduce the inequality
\[I_2\lesssim [w]_{A_{r,\infty}}\norm{b}_{\Lambda(\delta)}^{m}\norm{fw}_{r}|B|^{\frac{\tilde{\delta}}{n}}\]
and, thus,
\[\frac{\norm{w\chi_B}_{\infty}}{|B|}\int_B\abs{T_{ \alpha,b}^{m}f(x)-a_{B}}dx\lesssim \norm{b}_{\Lambda(\delta)}^{m}\norm{fw}_{r}|B|^{\frac{\widetilde{\delta}}{n}},\]
so it remains to take supremum over all the balls $B$ to get the desired result.
\end{proof}

\begin{proof}[Proof of Theorem \ref{limitcase} ]
    Let $B\subset \mathbb{R}^{n}$ be a ball and $x\in B$. Let $f=f_{1}+f_{2}$ with $f_{1}=f\chi_{2B}$. Then,
    \begin{align*}
    T_{\alpha,b}^{m}f(x)-(T_{\alpha,b}^{m}f)_{B}=&T_{\alpha,b}^{m}f_{1}(x)-(T_{\alpha,b}^{m}f_{1})_{B}\\
    &+\sum_{k=0}^{m}c_{k}\left[(b(x)-b_{B})^{m-k}T_{\alpha}((b-b_{B})^{k}f_{2})(x)\right.\\
    &\left.-\frac{1}{|B|}\int_{B}(b(z)-b_{B})^{m-k}T_{\alpha}((b-b_{B})^{k}f_{2})(z)dz\right].
    \end{align*}
    We can rewrite the above identity in the following form
    \begin{align*}
    T_{\alpha,b}^{m}f(x)&-(T_{\alpha,b}^{m}f)_{B}=\sigma_{1}(x)-(\sigma_{1})_{B}\\
    &\quad+\sum_{k=0}^{m}c_{k}\left[\sigma_{2}(x,u,k)-(\sigma_{2}(\cdot,u,k))_{B}+\sigma_{3}(x,u,k)-(\sigma_{3}(\cdot,u,k))_{B}\right],
    \end{align*}
    where
    \begin{equation*}
    \sigma_{1}(x)=T_{\alpha,b}^{m}f_{1}(x),
    \end{equation*}
    \begin{equation*}
    \sigma_{2}(x,u,k)=(b(x)-b_{B})^{m-k}\left(T_{\alpha}((b-b_{B})^{k}f_{2})(x)-T_{\alpha}((b-b_{B})^{k}f_{2})(u)\right),
    \end{equation*}
    \begin{equation*}
    \sigma_{3}(x,u,k)=(b(x)-b_{B})^{m-k}  T_{\alpha}((b-b_{B})^{k}f_{2})(u).
    \end{equation*}

    For $\sigma_{1}$, since $w\in A_{\frac{n}{m\delta+\alpha},\infty}$, there exists $1<p<\frac{n}{m\delta+\alpha}$ such that $w\in A_{p,\infty}$. We take $\frac{1}{q}=\frac{1}{p}-\frac{m\delta+\alpha}{n}$, so $q>p$ and $w\in A_{q,\infty}$ and, moreover, $w\in A_{p,q}$. By applying H\"older's inequality with $q$ and $q'$ and the boundedness of $T^{m}_{\alpha,b}$ from $L^{p}_{w}(\mathbb R^n)$ to $L^{q}_{w}(\mathbb R^n)$ (Theorem \ref{teo1}) we obtain that
    \begin{align*}
    \frac{1}{|B|}\int_{B}|\sigma_{1}(x)|dx&\leq \frac{C}{|B|}\left(\int_{B}\abs{T_{\alpha,b}^{m}f_{1}(x)w(x)}^{q}\right)^{1/q}
    \norm{w^{-1}\chi_{B}}_{q'}\\
    &\lesssim\|b\|_{\Lambda(\delta)}^m\frac{\norm{fw\chi_{B}}_{p}}{|B|}\norm{w^{-1}\chi_{B}}_{q'}.
    \end{align*}
    Since $\frac{1}{p}=\frac{\alpha+(m-1)\delta}{n}+\frac{1}{q}+\frac{\delta}{n}$, we can apply again H\"older's inequality and the fact that $w\in A_{q, \infty}$ to get
    \begin{align*}
    \frac{1}{|B|}\int_{B}|\sigma_{1}(x)|dx
    &\lesssim\|b\|_{\Lambda(\delta)}^m\norm{fw}_{\frac{n}{\alpha+(m-1)\delta}}|B|^{\delta/n}\frac{\norm{w^{-1}\chi_{B}}_{q'}}{|B|^{1/q'}}\\
    &\lesssim\|b\|_{\Lambda(\delta)}^m\norm{fw}_{\frac{n}{\alpha+(m-1)\delta}}|B|^{\delta/n}\norm{w\chi_{B}}_{\infty}^{-1}.
    \end{align*}

    In order to estimate $\sigma_{2}$ we use the inequality \[\frac{1}{|B|}\int_{B}\abs{b(x)-b_{B}}^{m-k}dx\leq \norm{b}_{\Lambda(\delta)}^{m-k}|B|^{\frac{\delta(m-k)}{n}}\]
    and Lemma \ref{lema m>1} to obtain
    \begin{align*}
    \frac{1}{|B|}\int_{B}\abs{\sigma_{2}(x,u,k)}dx&\lesssim\norm{b}^{k}_{\Lambda(\delta)}\norm{fw}_{\frac{n}{\alpha+(m-1)\delta}}\norm{w\chi_{B}}^{-1}_{\infty}\frac{
    |B|^{\frac{\delta(k-m+1)}{n}}}{|B|}\int_{B}\abs{b(x)-b_{B}}^{m-k}dx\\
    &  \lesssim\norm{b}^{m}_{\Lambda(\delta)}\norm{fw}_{\frac{n}{\alpha+(m-1)\delta}}\norm{w\chi_{B}}^{-1}_{\infty}
    |B|^{\delta/n}.
    \end{align*}
    Consequently, since
    \begin{align*}
    \sum_{k=0}^{m}c_{k}\left[\sigma_{3}(x,u,k)-(\sigma_{3}(\cdot,u,k))_{B}\right]&=
    \left[T_{\alpha,b}^{m}f(x)-(T_{\alpha,b}^{m}f)_{B}\right]-\left[\sigma_{1}(x)-(\sigma_{1})_{B}\right]\\
    &-  \sum_{k=0}^{m}c_{k}\left[\sigma_{2}(x,u,k)-(\sigma_{2}(\cdot,u,k))_{B}\right]
    \end{align*}
    by first assuming that $T^{m}_{\alpha,b}f: L^{\frac{n}{\alpha+(m-1)\delta}}_{w} \hookrightarrow \mathbb{L}_{w}(\delta)$, then
    \begin{align*}
    \frac{1}{|B|}\int_{B}&\abs{\sum_{k=0}^{m}c_{k}\left[\sigma_{3}(x,u,k)-(\sigma_{3}(\cdot,u,k))_{B}\right]dx}\\
    &\leq
    \frac{1}{|B|}\int_{B} \abs{T_{\alpha,b}^{m}f(x)-(T_{\alpha,b}^{m}f)_{B}}dx+\frac{2}{|B|}\int_{B}|\sigma_{1}(x)|dx\\
    &\quad+  \sum_{k=0}^{m}c_{k}\frac{2}{|B|}\int_{B}\abs{\sigma_{2}(x,u,k)}dx\\
    &\lesssim\|b\|_{\Lambda(\delta)}^m\norm{fw}_{\frac{n}{\alpha+(m-1)\delta}}\norm{w\chi_{B}}_{\infty}^{-1}|B|^{\delta/n}.
    \end{align*}
    On the other hand, if we suppose that \eqref{ecu2-m>1} holds, it is easy to see that
    $T^{m}_{\alpha,b}f:L^{\frac{n}{\alpha+(m-1)\delta}}_{w}(\mathbb R^n)\hookrightarrow \mathbb{L}_{w}(\delta)$.
\end{proof}

     \begin{proof}[Proof of Theorem \ref{teo1H}:] We proceed by induction. We must point out that the case $m=0$ was already proved in \cite{BDP}. As in the proof of Theorem \ref{teo1} we have that
    \begin{equation*}
    \|wT^{m}_{\alpha,b}f\|_{q}\lesssim \|w M_{0,\gamma}^{\sharp}(T^{m}_{\alpha,b}f)\|_{q}.
    \end{equation*}
        We shall now use the second part of Lemma \ref{pointwise}, since we have that $K_\alpha \in S_{\alpha}\cap H_{\alpha,\Phi}$. Thus, we obtain that
        {\begin{align*}
        \|wT_{\alpha,b}^m f\|_{q}   &\lesssim \sum\limits_{j=0}^{m-1} \|wM_{\theta_j,\gamma} (|T_{\alpha,b}^j f|)\|_{q}+\|wM_{m\delta+\alpha,\widetilde{\Phi}}f\|_{q},
        \end{align*}
        where we have assumed, without loss of generality, that $\|b\|_{\Lambda(\delta)}=1$.}

        From the hypothesis on the weight $w$ and the Young function $\widetilde{\Phi}$, by Theorem \ref{BDPcteH} we know that $\|wM_{m\delta+\alpha,\widetilde{\Phi}}f\|_{q}\lesssim \|fw\|_p$.

        The proof now follows in the same way as in the proof of Theorem \ref{teo1}.
    \end{proof}

\begin{proof}[Proof of Theorem \ref{HormanderLipw}:]
        Take $f, f_1, f_2$ and $a_B$ as in the proof of Theorem \ref{boundednessLipw}, and define $I_1$ and $I_2$ likewise.

        Since in $I_1$ we have only used the size condition $S_\alpha$, the estimation is the same, by taking into account that $w^\beta\in A_{r/\beta,\infty}$ yields $w\in A_{r,\infty}$ for any $\beta\geq 1$.

        For $I_2$ we proceed similarly but we have to use now that $K_\alpha\in H_{\alpha, \Phi, m}(\delta)$ with $\Phi^{-1}(t)\lesssim t^{\frac{\beta-1}{r}}$ for some $1<\beta<r$ and all $t>0$. We split the average into $I_3$ and $I_4$ as in the proof of Theorem \ref{boundednessLipw}. The last one can be controlled in the same form. The difference will be in $I_3$. Recall that
        \[I_3=\int_{(2B)^c}\abs{b(x)-b(z)}^{m}\abs{K_\alpha(x-z)-K_\alpha(y-z)}\abs{f(z)}dz,\]
for $x\in B$.

         By the definition of $\Lambda (\delta)$, we get that
        \begin{align*}
        I_{3}&\lesssim \norm{b}_{\Lambda(\delta)}^{m}\int_{(2B)^c}\abs{x-z}^{\delta m}\abs{K_{\alpha}(x-z)-K_{\alpha}(y-z)}\abs{f(z)}dz\\
        &\lesssim \norm{b}_{\Lambda(\delta)}^{m}\sum_{j=1}^{\infty}|2^{j+1}B|^{\frac{\delta m}{n}}
        \int_{2^{j+1}B\setminus 2^{j}B}\abs{K_{\alpha}(x-z)-K_{\alpha}(y-z)}\abs{f(z)}dz.
            \end{align*}
Now, since $K_\alpha\in H_{\alpha,\Phi,m}(\delta)$, $w^\beta\in A_{r/\beta,\infty}$ and $\Phi^{-1}(t)t^{1/r}t^{1/(r/\beta)'}\lesssim t$, we can prodeed as in \eqref{lemma3.9} with $k=m$ to obtain
            \[I_3\lesssim \frac{\norm{b}_{\Lambda(\delta)}^{m}|B|^{\frac{\widetilde{\delta}}{n}}\|fw\|_r}{\|w\chi_B\|_\infty}.\qedhere\]
\end{proof}

\begin{proof}[Proof of Theorem \ref{limitcaseH}:]
        We proceed as in the proof of Theorem \ref{limitcase}. We must only use the corresponding hypothesis on the kernel, that guarantees the validity of Theorem \ref{teo1H} and Lemma \ref{limitcaseH}, which are immediate from the fact that $S_\alpha\cap H_{\alpha, \Phi,m}(\delta)\subset S_\alpha\cap H_{\alpha,\Phi}$ (see Remark \ref{contentionsHdelta}).
    \end{proof}

     \bibliographystyle{abbrv}

\begin{thebibliography}{10}
    	
    	\bibitem{BDP}
    	A.~Bernardis, E.~Dalmasso, and G.~Pradolini.
    	\newblock Generalized maximal functions and related operators on weighted
    	{M}usielak-{O}rlicz spaces.
    	\newblock {\em Ann. Acad. Sci. Fenn. Math.}, 39(1):23--50, 2014.
    	
    	\bibitem{BHP}
    	A.~Bernardis, S.~Hartzstein, and G.~Pradolini.
    	\newblock Weighted inequalities for commutators of fractional integrals on
    	spaces of homogeneous type.
    	\newblock {\em J. Math. Anal. Appl.}, 322(2):825--846, 2006.
    	
    	\bibitem{BLR}
    	A.~L. Bernardis, M.~Lorente, and M.~S. Riveros.
    	\newblock Weighted inequalities for fractional integral operators with kernel
    	satisfying {H}\"ormander type conditions.
    	\newblock {\em Math. Inequal. Appl.}, 14(4):881--895, 2011.
    	
    	\bibitem{BCe}
    	M.~Bramanti and M.~C. Cerutti.
    	\newblock Commutators of singular integrals and fractional integrals on
    	homogeneous spaces.
    	\newblock 189:81--94, 1995.
    	
    	\bibitem{BCM}
    	M.~Bramanti, M.~C. Cerutti, and M.~Manfredini.
    	\newblock {$L^p$} estimates for some ultraparabolic operators with
    	discontinuous coefficients.
    	\newblock {\em J. Math. Anal. Appl.}, 200(2):332--354, 1996.
    	
    	\bibitem{Chanillo}
    	S.~Chanillo.
    	\newblock A note on commutators.
    	\newblock {\em Indiana Univ. Math. J.}, 31(1):7--16, 1982.
    	
    	\bibitem{CWW}
    	S.~Chanillo, D.~K. Watson, and R.~L. Wheeden.
    	\newblock Some integral and maximal operators related to starlike sets.
    	\newblock {\em Studia Math.}, 107(3):223--255, 1993.
    	
    	\bibitem{ChFL}
    	F.~Chiarenza, M.~Frasca, and P.~Longo.
    	\newblock Interior {$W^{2,p}$} estimates for nondivergence elliptic equations
    	with discontinuous coefficients.
    	\newblock {\em Ricerche Mat.}, 40(1):149--168, 1991.
    	
    	\bibitem{ChFL2}
    	F.~Chiarenza, M.~Frasca, and P.~Longo.
    	\newblock {$W^{2,p}$}-solvability of the {D}irichlet problem for nondivergence
    	elliptic equations with {VMO} coefficients.
    	\newblock {\em Trans. Amer. Math. Soc.}, 336(2):841--853, 1993.
    	
    	\bibitem{CUF}
    	D.~Cruz-Uribe and A.~Fiorenza.
    	\newblock Endpoint estimates and weighted norm inequalities for commutators of
    	fractional integrals.
    	\newblock {\em Publ. Math.}, 47(1):103--131, 2003.
    	
    	\bibitem{DP}
    	E.~Dalmasso and G.~Pradolini.
    	\newblock Characterizations of the boundedness of generalized fractional
    	maximal functions and related operators in {O}rlicz spaces.
    	\newblock {\em Math. Nach.}, 290(1):19--36, 2017.
    	
    	\bibitem{DR}
    	L.~Diening and M.~R{\r{u}}{\v{z}}i{\v{c}}ka.
    	\newblock Calder\'on-{Z}ygmund operators on generalized {L}ebesgue spaces
    	${L}^{p(\cdot)}$ and problems related to fluid dynamics.
    	\newblock {\em J. Reine Angew. Math.}, 563:197--220, 2003.
    	
    	\bibitem{DL}
    	Y.~Ding and S.~Z. Lu.
    	\newblock Weighted norm inequalities for fractional integral operators with
    	rough kernel.
    	\newblock {\em Canad. J. Math.}, 50(1):29--39, 1998.
    	
    	\bibitem{DLZ}
    	Y.~Ding, S.~Z. Lu, and P.~Zhang.
    	\newblock Weak estimates for commutators of fractional integral operators.
    	\newblock {\em Sci. China Math. (Ser. A)}, 44(7):877--888, 2001.
    	
    	\bibitem{GCRF}
    	J.~Garc{\'i}a-Cuerva and J.~L.~R. de~Francia.
    	\newblock {\em Weighted norm inequalities and related topics}, volume 116 of
    	{\em North-Holland Mathematics Studies}.
    	\newblock North-Holland Publishing Co., Amsterdam, 1985.
    	\newblock Notas de Matem{\'a}tica [Mathematical Notes], 104.
    	
    	\bibitem{GCHST}
    	J.~Garc{\'{\i}}a-Cuerva, E.~Harboure, C.~Segovia, and J.~L. Torrea.
    	\newblock Weighted norm inequalities for commutators of strongly singular
    	integrals.
    	\newblock {\em Indiana Univ. Math. J.}, 40(4):1397--1420, 1991.
    	
    	\bibitem{GPS4}
    	O.~Gorosito, G.~Pradolini, and O.~Salinas.
    	\newblock Weighted weak-type estimates for multilinear commutators of
    	fractional integrals on spaces of homogeneous type.
    	\newblock {\em Acta Math. Sin. (Engl. Ser.)}, 23(10):1813--1826, 2007.
    	
    	\bibitem{HSV}
    	E.~Harboure, O.~Salinas, and B.~Viviani.
    	\newblock Orlicz boundedness for certain classical operators.
    	\newblock {\em Colloq. Math.}, 91(2):263--282, 2002.
    	
    	\bibitem{HST}
    	E.~Harboure, C.~Segovia, and J.~L. Torrea.
    	\newblock Boundedness of commutators of fractional and singular integrals for
    	the extreme values of {$p$}.
    	\newblock {\em Illinois J. Math.}, 41(4):676--700, 1997.
    	
    	\bibitem{HL}
    	G.~H. Hardy and J.~E. Littlewood.
    	\newblock Some properties of fractional integrals. {II}.
    	\newblock {\em Math. Z.}, 34(1):403--439, 1932.
    	
    	\bibitem{JN1}
    	F.~John and L.~Nirenberg.
    	\newblock On functions of bounded mean oscillation.
    	\newblock {\em Comm. Pure Appl. Math.}, 14:415--426, 1961.
    	
    	\bibitem{Ku}
    	D.~S. Kurtz.
    	\newblock Sharp function estimates for fractional integrals and related
    	operators.
    	\newblock {\em J. Austral. Math. Soc. A}, 49:129--137, 1990.
    	
    	\bibitem{LK}
    	A.~Lerner and A.~Karlovich.
    	\newblock Commutators of singular integrals on generalized ${L}^p$ spaces with
    	variable exponent.
    	\newblock {\em Publ. Math.}, 49(1):111--125, 2005.
    	
    	\bibitem{LMRT}
    	M.~Lorente, J.~M. Martell, M.~S. Riveros, and A.~de~la Torre.
    	\newblock Generalized {H}\"ormander's conditions, commutators and weights.
    	\newblock {\em J. Math. Anal. Appl.}, 342(2):1399--1425, 2008.
    	
    	\bibitem{MY}
    	Y.~Meng and D.~Yang.
    	\newblock Boundedness of commutators with {L}ipschitz functions in
    	non-homogeneous spaces.
    	\newblock {\em Taiwanese J. Math.}, 10(6):1443--1464, 2006.
    	
    	\bibitem{MW}
    	B.~Muckenhoupt and R.~Wheeden.
    	\newblock Weighted norm inequalities for fractional integrals.
    	\newblock {\em Trans. Amer. Math. Soc.}, 192:261--274, 1974.
    	
    	\bibitem{P}
    	C.~P\'{e}rez.
    	\newblock Endpoint estimates for commutators of singular integral operators.
    	\newblock {\em J. Funct. Anal.}, 128(1):163--185, 1995.
    	
    	\bibitem{P2}
    	C.~P\'{e}rez.
    	\newblock Sharp estimates for commutators of singular integrals via iterations
    	of the {H}ardy-{L}ittlewood maximal function.
    	\newblock {\em J. Fourier Anal. Appl.}, 3(6):743--756, 1997.
    	
    	\bibitem{PPTT}
    	C.~P\'{e}rez, G.~Pradolini, R.~H. Torres, and R.~Trujillo-González.
    	\newblock End-point estimates for iterated commutators of multilinear singular
    	integrals.
    	\newblock {\em Bull. Lond. Math. Soc.}, 46(1):26--42, 2014.
    	
    	\bibitem{Pra}
    	G.~Pradolini.
    	\newblock Two-weighted norm inequalities for the fractional integral operator
    	between {$L^p$} and {L}ipschitz spaces.
    	\newblock {\em Comment. Math. (Prace Mat.)}, 41:147--169, 2001.
    	
    	\bibitem{Pradopolonia}
    	G.~Pradolini.
    	\newblock Two-weighted norm inequalities for the fractional integral operator
    	between {$L^p$} and {L}ipschitz spaces.
    	\newblock {\em Comment. Math. (Prace Mat.)}, 41:147--169, 2001.
    	
    	\bibitem{PR2}
    	G.~Pradolini and W.~Ramos.
    	\newblock Characterization of {L}ipschitz functions via the commutator of
    	singular and fractional integral operators in variable {L}ebesgue spaces.
    	\newblock {\em Potential Anal.}, 2016.
    	
    	\bibitem{PS}
    	G.~Pradolini and O.~Salinas.
    	\newblock The fractional integral between weighted {O}rlicz and ${BMO}_\phi$
    	spaces on spaces of homogeneous type.
    	\newblock {\em Comment. Math. Univ. Carolin.}, 44(3):469--487, 2003.
    	
    	\bibitem{Rios}
    	C.~Rios.
    	\newblock The {$L^p$} {D}irichlet problem and nondivergence harmonic measure.
    	\newblock {\em Trans. Amer. Math. Soc.}, 355(2):665--687, 2003.
    	
    	\bibitem{ST}
    	C.~Segovia and J.~L. Torrea.
    	\newblock Weighted inequalities for commutators of fractional and singular
    	integrals.
    	\newblock {\em Publ. Mat.}, 35(1):209--235, 1991.
    	\newblock Conference on Mathematical Analysis (El Escorial, 1989).
    	
    	\bibitem{Swe}
    	E.~M. Stein and G.~Weiss.
    	\newblock Fractional integrals on {$n$}-dimensional {E}uclidean space.
    	\newblock {\em J. Math. Mech.}, 7:503--514, 1958.
    	
    \end{thebibliography}

\end{document}